\documentclass[12pt,reqno]{article}

\usepackage{amsmath,amsfonts,amsthm,amssymb,amsxtra}
\usepackage[numbers]{natbib}

\usepackage{color}
\usepackage{graphicx,subfigure}
\usepackage{multirow}
\usepackage{epstopdf}
\usepackage{graphicx,subfigure}
\usepackage{multirow}
\usepackage{epstopdf}

\usepackage{amsmath,latexsym,amssymb,amsfonts,amsbsy}
\usepackage{amsthm}
\usepackage{mathrsfs}
\usepackage{graphicx,subfigure}
\usepackage{multirow}
\usepackage{epstopdf}
\usepackage{bm}
\usepackage{mathtools}
\usepackage{array}
\usepackage{algorithm}
\usepackage{algorithmic}
\usepackage{palatino}
\usepackage{booktabs}
\usepackage{caption}

\usepackage{marginnote}

\usepackage{ulem}
\usepackage{cancel}

\usepackage{soul}

\usepackage[colorlinks,citecolor=red,pagebackref,hypertexnames=false]{hyperref}

\newtheorem{theorem}{Theorem}

\newtheorem{proposition}[theorem]{Proposition}
\newtheorem{lemma}[theorem]{Lemma}
\newtheorem{corollary}[theorem]{Corollary}

\theoremstyle{definition}

\newtheorem{definition}[theorem]{Definition}

\theoremstyle{remark}

\newtheorem{remark}[theorem]{Remark}

\renewcommand{\epsilon}{\varepsilon}

\renewcommand{\phi}{\varphi}
\newcommand{\R}{\mathbb{R}}

\newcommand{\Z}{\mathbb{Z}}
\newcommand{\om}{\omega}
\newcommand{\Om}{\Omega}

\newcommand{\E}{\mathbb{E}}
\newcommand{\lan}{\langle}
\newcommand{\ran}{\rangle}

\newcommand{\wt}{\widetilde}
\newcommand{\eps}{\varepsilon}
\renewcommand{\P}{\mathbb{P}}

\newcommand{\ipc}[2]{\left \langle #1 , \ #2 \right \rangle }

\newcommand{\vertiii}[1]{{\left\vert\kern-0.25ex\left\vert\kern-0.25ex\left\vert #1
    \right\vert\kern-0.25ex\right\vert\kern-0.25ex\right\vert}}

\DeclareMathOperator{\Bern}{Bern}

\makeatletter
\def\blfootnote{\gdef\@thefnmark{}\@footnotetext}
\makeatother

\begin{document}

\title{Approximating the ground state eigenvalue via the effective potential}
\author{Ilias Chenn, Wei Wang and Shiwen Zhang}
\date{  }
\maketitle

\newcommand{\Addresses}{{
  \bigskip
  \vskip 0.08in \noindent --------------------------------------

\footnotesize
\medskip

I.~Chenn, \textsc{Department of Mathematics, Massachusetts Institute of Technology, 2-252b, 77 Massachusetts Avenue, 
Cambridge, MA  4307 USA }\par\nopagebreak
  \textit{E-mail address}: \texttt{nehcili@mit.edu}
 
\vskip 0.4cm

  W. ~Wang, \textsc{School of Mathematics, University of Minnesota, 206 Church St SE, Minneapolis, MN 55455 USA}\par\nopagebreak
  \textit{E-mail address}: \texttt{wang9585@umn.edu }
  
\vskip 0.4cm

S.~Zhang, \textsc{School of Mathematics, University of Minnesota, 206 Church St SE, Minneapolis, MN 55455 USA}\par\nopagebreak
  \textit{E-mail address}: \texttt{zhan7294@umn.edu }
  }}

\begin{abstract}
In this paper, we study 1-d random Schr\"odinger operators on a finite interval with  Dirichlet boundary conditions.  We are interested in the approximation of the ground state energy using the minimum of the  effective potential. For the 1-d continuous Anderson Bernoulli model, we show that the ratio of the ground state energy and the minimum of the effective potential approaches  $\frac{\pi^2}{8}$ as the domain size approaches infinity. Besides, we will discuss various approximations to the ratio in different situations. There will be numerical experiments supporting our main results for the ground state energy and also supporting approximations for the excited states energies.  
\end{abstract}

\tableofcontents

\section{Introduction}
In \cite{FM}, Filoche and Mayboroda introduced the concept of localization landscape function, which is a solution $u$ to $Hu=1$ for an elliptic operator $H$.
In  \cite{FM} and a series of companion papers \cite{ADFJM-PRL,ADFJM-CPDE,ADFJM-SIAM}, the authors used the landscape function $u$ and its reciprocal $1/u$, the so called effective potential, to predict eigenvalues and eigenfunctions of $H$ without explicitly solving the eigenvalue problem. 
For a Schr\"{o}dinger operator  $H= -\Delta + V$ with a nonegative Anderson type potential $V$ on some bounded domain $\Omega\subset\R^d$,  denote by $u$ the  associated landscape function of $H$ and by $\lambda_n$ the $n$-th smallest eigenvalue of $H$.
Arnold et. al. observed in \cite{ADFJM-SIAM} that
\begin{equation}
    \frac{\lambda_n}{\left(\min \frac{1}{u}\right)_n}  \approx  1+ \frac{d}{4} , \label{eq:obs}
\end{equation}
where $\left(\min \frac{1}{u}\right)_n$ is the $n$-th local minimum of $1/u$ on $\Om$. \cite{ADFJM-SIAM} provided convincing numerical evidence and heuristic arguments to support \eqref{eq:obs}. 

In this paper, we focus on a  1-d Schr\"{o}dinger operator $H= -\Delta + V$ on a finite domain (interval), with a piecewise constant Anderson type potential $V$(see the precise definition in \eqref{eq:V}). We provide a detailed study of the observation \eqref{eq:obs} for the ground state energy case $n=1$.   More precisely, we studied the asymptotic behavior of the quantity $\dfrac{\lambda_1}{\min \frac{1}{u}}$ either as the domain size or the strength of the potential approaches infinity.  In particular, we show that
\begin{equation}\label{eq:first}
\dfrac{\lambda_1}{\min \frac{1}{u}}\approx\dfrac{\pi^2}{8}
\end{equation}
in either case for the  Anderson Bernoulli model. 
Moreover, we will infer similar results of  $\lambda_n/\left(\min \frac{1}{u}\right)_n$ for the excited states energies case $n\ge2$ by numerical means. We may apply these results to predict eigenvalues $\lambda_n$ at the bottom of the spectrum, using the local minima $\left(\min \frac{1}{u}\right)_n$. 

	 \begin{figure}[H]
	 	\centerline
	 	{	\includegraphics[width=1.2\textwidth]{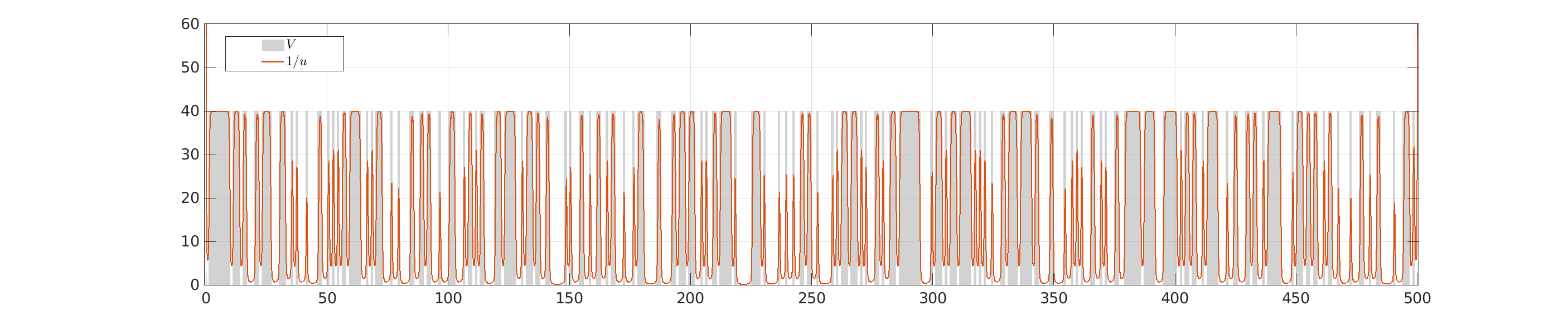}	}
	 	\caption{A Bernoulli potential $V$ and the associated effective potential $1/u$. $V$: 50\% 0 and  50\% 40 on the domain [0,500]}
	 	\label{fig1}
	 \end{figure}
	 \begin{figure}[H]
	 	\centerline{
	 		\includegraphics[width=0.5\textwidth]{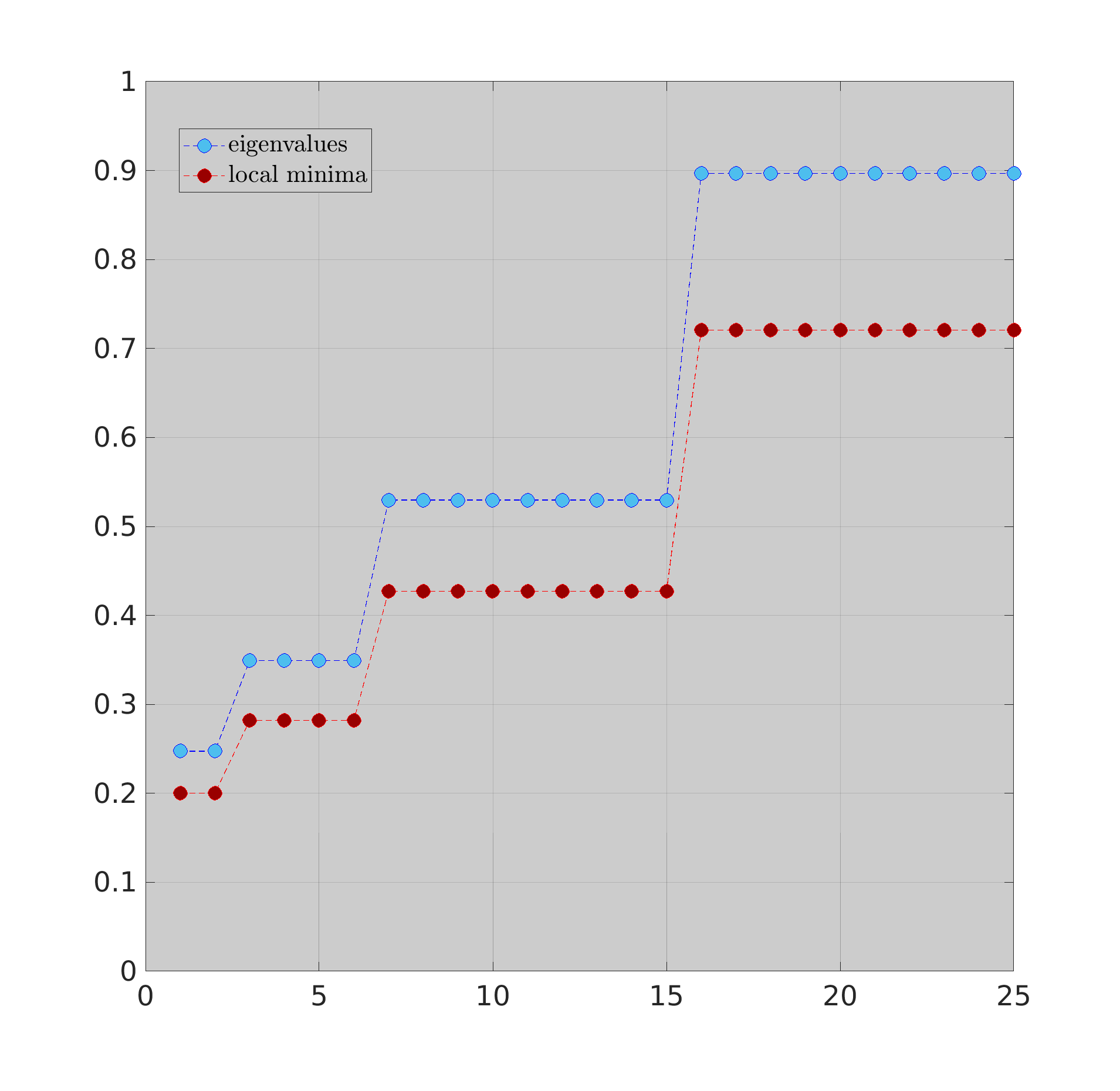}
	 		\includegraphics[width=0.5\textwidth]{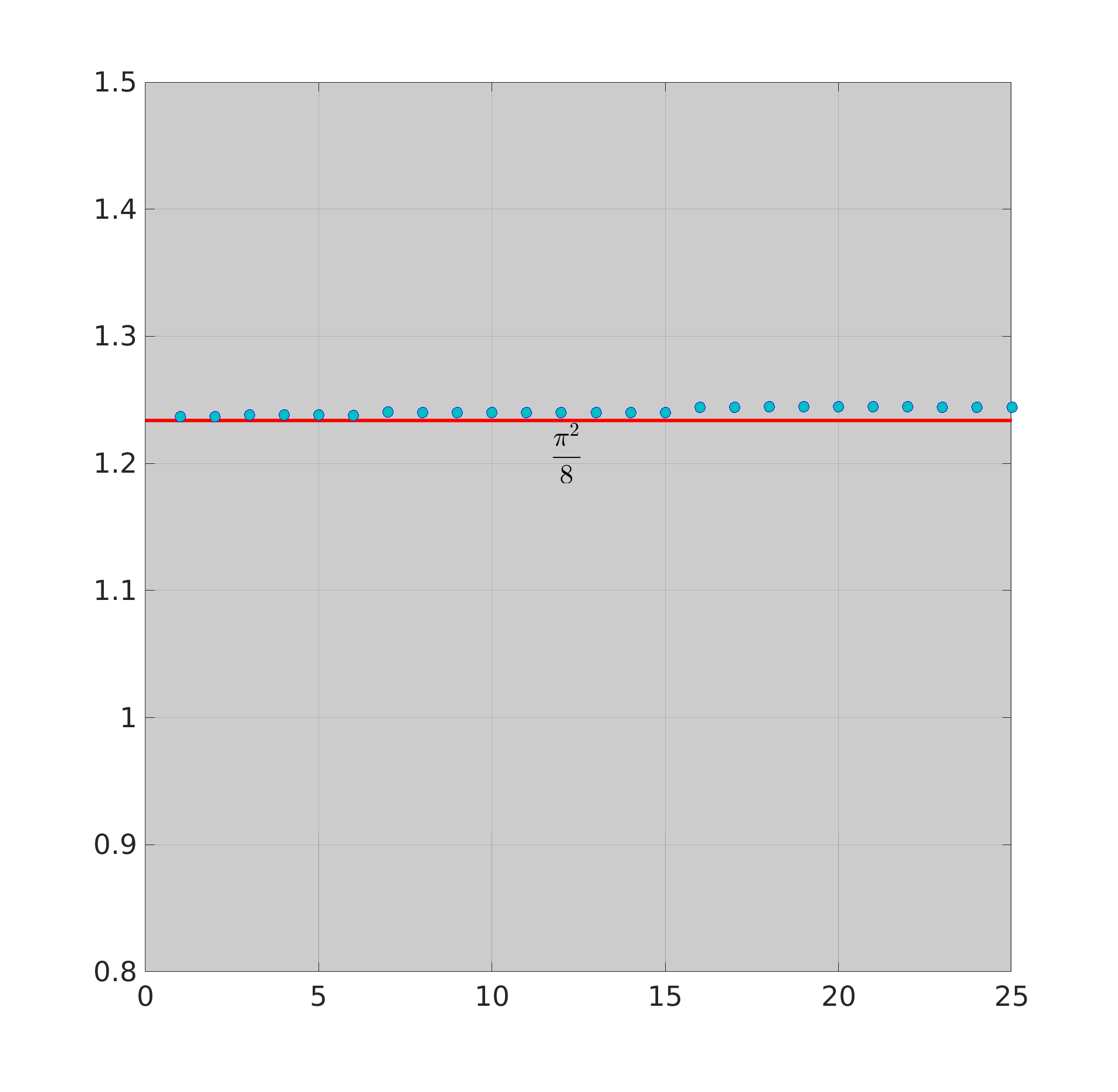}
	 	}	
	 	\caption{ For the potential $V$ in Figure \ref{fig1}, the left figure displays a comparison of the first 25 eigenvalues with the corresponding local minima.  The right figure displays the corresponding ratio $\lambda_n/\left(\min \frac{1}{u}\right)_n $, $n=1,2,\cdots,25$, and the horizontal reference line $\frac{\pi^2}{8}$ .}\label{fig:2}	
	 \end{figure}

Before we state our main results in the next section, let us discuss more background and related works. The simplest case of \eqref{eq:obs} is that of the relation between the ground state energy $\lambda_1$ and $\min \frac{1}{u}$, as we intend to study in this paper.   
The landscape function $u=H^{-1}1$ is also known as the torsion function in many other contexts, see e.g. \cite{vC,Vogt} and references therein. 
A recent work due to Vogt  \cite{Vogt}, leveraging on \cite{Ou}, provides a quantitative bound for \eqref{eq:obs} for the ground state $\lambda_1$ in the form
\begin{equation}\label{eq:Vogt}
    1 \leq \frac{\lambda_1}{\min \frac{1}{u}} \leq 1 + d/8 + c d^{1/2},
\end{equation}
where the explicit constant $c\approx 0.6055$. 
The bounds \eqref{eq:Vogt} hold for a large class of operators, but are not optimal as was remarked by the author in \cite{Vogt}. Indeed, in dimension $d=1$, for the free Hamiltonian $-\Delta$ on an interval $[0,L]$ with Dirichlet boundary conditions, the ground state energy $\lambda_1$ and the landscape function $u$ can be computed explicitly:
\begin{equation*}
\lambda_1=\frac{\pi^2}{L^2},\quad {\rm and} \quad u=\frac{1}{2}x(L-x), 
\end{equation*}
which implies  
\begin{equation}\label{eq:free}
    \frac{\lambda_1}{\min \frac{1}{u}} = \frac{\pi^2}{8} \approx 1.23 < 1+\frac{1}{4}.
\end{equation}

As we see in Figure \ref{fig:2}, the ratio $\pi^2/8$, given by the ground state of the free system \eqref{eq:free}, predicts the asymptotic behavior of ${\lambda_n}/({\min \frac{1}{u}})_n$ for the Anderson model $H=-\Delta+V$. We will provide the rigorous proof of this  prediction for the ground state energy case $n=1$ and more numerical experiments supporting the prediction for the excited states energies case $n\ge 2$.

Throughout the paper,  we will denote by $C_{i}$ some finite constants. For simplicity, $C$ or $C_{i}$ may stand for different constants simultaneously. We will write $ A\lesssim B$, $B\gtrsim A$, or $A = O(B)$ if $A \leq CB$ for some constant $C$.  Lastly, we will write $A \approx B$ if $A \lesssim B$ and $B \lesssim A$.  

The remainder of this paper is organized as follows. 
In Section \ref{sec:theory}, we record our main results. We include some preliminaries and prove the main results in Section \ref{sec:proof}. Section \ref{sec:numerics} presents numerical results supporting the analysis.

\section{Main results} \label{sec:theory}
In this paper, we will be concerned with the 1-d random potential 
\begin{equation}\label{eq:V}
    V_{\om}=V_{\om}(x)=\sum_{j\in\Z}\omega_j\chi(x-j)\ \ {\rm for}\ \ x\in \R,
\end{equation}
 where $\chi(x)$ is the characteristic function of $[0,1)$, and $\{\omega_j\}_{i\in\Z}$ are nonegative, independent and identically distributed (i.i.d.)
random variables on a probability space $(\Theta,{\mathcal F},\P)$. Throughout the paper, we will refer to \eqref{eq:V} a piecewise constant Anderson type potential.

For simplicity, we will call such $V_{\om}$ an $\om$-piecewise  potential, where $\om$ obeys the common distribution of $\om_n$. 
We consider the 1-d random Schr\"odinger operator on $\Omega=(0,L)$
\begin{equation}\label{eq:H}
    H=-\Delta+k V_{\om},
\end{equation} 
where $V_{\om} $ is a nonegative $\om$-piecewise potential as in \eqref{eq:V}, and $k\ge 0$ is a coupling constant measuring the strength of the potential. 
Such $H$ with an Anderson type potential $V_{\om}$ is a typical class of the one-dimensional (continuous) alloy-type Anderson model. 
In this paper, we will simply call $H$ the Anderson Bernoulli model if $\om_j$ are i.i.d. Bernoulli random variables. We restrict our scope to a smaller subset of Anderson model with potentials $V_{\om}$ in the piecewise constant form \eqref{eq:V} for the sake of clarity in our theoretical treatment. These potentials capture the main features of random potentials while being readily used in models of semi-conductor simulations. We refer readers to a more detailed introduction to more general Anderson models and  alloy-type potentials in e.g. \cite{AW,DSS,Kir,KM1} and references therein.

We always assume the domain size $L$ is a positive integer for simplicity. We are interested in the Dirichlet eigenvalue problem of $H$ on $\Omega$. Since $kV_\om \ge 0$, the ground state eigenvalue of $H$, denoted by $\lambda_1$, is always strictly positive. The landscape function $u$, the solution to the Dirichlet boundary problem 
\[Hu=1\ \  {\rm on}\ \Omega, \ \ u(0)=u(L)=0, \]
exists and is unique. Moreover, $u>0$ by the maximum principle. See more about the landscape function in Section \ref{sec:pre}. 

We will study the asymptotic behavior of the quantity $\lambda_1/(\min \frac{1}{u})$, or equivalently $\lambda_1 \max u$ as $L$ or $k$ varies, where $\max u$ is the maximum of $u$ on $[0,L]$.
In each one of our results, there is a competition between 
the strength of the disorder and certain characteristic size of  spatial length.

The first result is
\begin{theorem}
\label{thm:1}
 
 Let $H=-\Delta+kV_{\om}$ be as in \eqref{eq:H} and  with a nonegative $\om$-piecewise potential $V_{\om}$ as in \eqref{eq:V}. Let $\lambda_1$ and $u$ be the ground state eigenvalue  and the landscape function of $H$ on $\Omega$ with   Dirichlet boundary conditions, respectively. 
 Suppose $\om$ is a nonegative random variable on $\R$ such that
\begin{equation}\label{eq:sing-F}
    0 < \P(\om = 0) < 1.
\end{equation} 
 For any positive integer $L$ and any realization of $\om$, let $L_{\max}=L_{\max}(L,\om)$ denote the longest length of an interval in $[0,L]$ on which $V_{\om}(x) = 0$. Suppose there are constants $C, \alpha>0$ such that $k$ satisfies
\begin{equation}
 kL_{\max}^{1-\alpha}>C \label{eq:k-beta}
\end{equation} 
for all sufficiently large $L$ and almost surely all $\om$. 
Then  
\begin{equation}
     \lim_{L \rightarrow \infty} \frac{\lambda_1}{\min \frac{1}{u}}   = \frac{\pi^2}{8} \label{eq:main-pi28}
  \end{equation} with probability one. 
 
 In particular, 
 for any fixed $k>0$, \eqref{eq:main-pi28} holds with probability one.
\end{theorem}

Theorem \ref{thm:1} shows that
\begin{equation*}
     \frac{\lambda_1}{\min \frac{1}{u}}  \approx \frac{\pi^2}{8} \approx 1.23
\end{equation*}
as $L \rightarrow \infty$. Since  $1+\frac{1}{4} = 1.25$, the observation made in \cite{ADFJM-SIAM}
\begin{equation*}
   \frac{\lambda_1}{\min \frac{1}{u}}\approx  1+\frac{d}{4}
\end{equation*}
 holds approximately in $d=1$. Though, the more accurate constant in the asymptotic regimes is in fact $\frac{\pi^2}{8}$. 

The first theorem considers the limit of $\lambda_1/  (\min \frac{1}{u})$ as the domain size $L\to\infty$. Alternatively, we may fix $L$ and consider the semi-classical limit for extremely large disorder $k$. We obtain the same limit $\frac{\pi^2}{8}$ as $k\to\infty$ if \eqref{eq:sing-F} holds. Moreover, we see a different limit if there is no atom at $0$ in the probability distribution. More precisely, we prove
\begin{theorem}[Semi-classical limit] \label{thm:2}
Let $H=-\Delta+kV_{\om}$ be as in \eqref{eq:H} with a nonegative $\om$-piecewise potential $V_{\om}$ as in \eqref{eq:V}. Let $p=\P(\om = 0)$. Fix any positive integer $L$. Then
\begin{equation}
       \lim_{k \rightarrow \infty}\,   \frac{\lambda_1}{\min \frac{1}{u}}  = \frac{\pi^2}{8}, \label{eq:semi-1}
\end{equation}
with probability $1-(1-p)^L$, and 
\begin{equation}\label{eq:semi-2}
    \lim_{k \rightarrow \infty} \frac{\lambda_1}{\min \frac{1}{u}} =1,
\end{equation}
with probability $(1-p)^L$. 

In particular, if $\P(\om = 0)=0$, then \eqref{eq:semi-2} holds with probability one. 
\end{theorem}
The analysis in this direction is very natural. If there is at least one zero well in the domain, then the walls created by the nonegative potential become  higher and higher as $k$ increases. The system is eventually decoupled into direct sum of (negative) free Laplacian on each zero well as $k\to \infty$, in which case we obtain the semi-classical limit \eqref{eq:semi-1} as in the free case. In the case $\inf V_{\om} >0$, $-\Delta+kV_{\om}$ behaves ``diagonally dominantly'' as $kV_{\om}$ on any finite domain as $k\to \infty$. Hence, $\lambda_1\approx \inf kV_{\om} \approx \min \frac{1}{u}$, which leads to\eqref{eq:semi-2}.  We include the detailed proof in Section \ref{sec:semi-classical}. 

Combing Theorem \ref{thm:1} and \ref{thm:2}, we see that for the Anderson Bernoulli model, the ratio ${\lambda_1}/{\min \frac{1}{u}}$ approaches $\dfrac{\pi^2}{8}$ either as the domain size $L$ or  the disorder strength $k$ approaches infinity. 
\begin{corollary}\label{cor:Ber}
Let $H=-\Delta+kV_{\om}$ be as in \eqref{eq:H} with an $\om$-piecewise potential $V_{\om}$ as in \eqref{eq:V}. Suppose $\om$ satisfies the $\{0,1\}$-Bernoulli distribution, i.e., $\P(\om=0)=p$ and $\P(\om=1)=1-p$ for some $p\in (0,1)$. Then for any fixed $k>0$, 
\begin{equation*}
     \lim_{L \rightarrow \infty} \frac{\lambda_1}{\min \frac{1}{u}}   = \frac{\pi^2}{8}
  \end{equation*} with probability one. And for any fixed positive integer $L$, 
\begin{equation*}
       \lim_{k \rightarrow \infty}\,   \frac{\lambda_1}{\min \frac{1}{u}}  = \frac{\pi^2}{8}
\end{equation*}
with probability $1-(1-p)^L$. 
\end{corollary}

Even though the observed constant $1+\dfrac{1}{4}$ in \cite{ADFJM-SIAM} is not accurate in view of the asymptotic behaviors in Theorem \ref{thm:1},\ref{thm:2}, we will show that the optimal proportionality constant can actually range from $1$ to $\dfrac{\pi^2}{8}$ for suitable finite domain size with relatively small disorder.  More precisely,

\begin{theorem}\label{thm:3}
Let $H=-\Delta+kV_{\om}$ be as in \eqref{eq:H} with an $\om$-piecewise potential $V_{\om}$ as in \eqref{eq:V}. Suppose $\omega$ is nonegative and bounded from above. 
For any $r\in [1,\dfrac{\pi^2}{8}]$ and any positive integer $L$, there is $k=k(r,L)$ such that 
\begin{equation}\label{eq:ratio-any}
  \lim_{L\to \infty}   \frac{\lambda_1}{\min \frac{1}{u}} =r 
\end{equation}
with probability one. 

In particular, for any sequence of $k=k(L)$ satisfying
\begin{equation}\label{eq:deloc-0}
    \lim_{L \rightarrow \infty} kL^2 = 0,
\end{equation}
one has 
\begin{equation}
    \lim_{L \rightarrow \infty} \frac{\lambda_1}{\min \frac{1}{u}} = \frac{\pi^2}{8} \label{eq:weak2}
\end{equation}
with probability one.  

For any sequence of $k=k(L)$ satisfying
\begin{equation}\label{eq:deloc-1}
    \lim_{L \rightarrow \infty} kL^2 = \infty\ \ {\rm and}\ \  \limsup_{L \rightarrow \infty} kL^{2-\beta}<\infty
\end{equation}
for some $0<\beta<1/4$, 
one has 
\begin{equation}
    \lim_{L \rightarrow \infty} \frac{\lambda_1}{\min \frac{1}{u}} = 1 \label{eq:weak1}
\end{equation}
with probability one. 
\end{theorem}

Notice in the special case \eqref{eq:weak2}, we obtained the same limit as in \eqref{eq:main-pi28}, without the singular assumption \eqref{eq:sing-F} on $\om$. The assumption \eqref{eq:deloc-0} is equivalent to the smallness condition on the disorder strength $k\ll L^{-2}\ll 1$. The limit \eqref{eq:weak2} is very natural since $H=-\Delta+kV_{\om}$ is now small perturbation of the negative free Laplacian $-\Delta$, as in \eqref{eq:free}.  We will give quantitative estimates for $\lambda_1$ and $u$ separately in Section \ref{sec:general}. Theorem \ref{thm:3} be proved as a direct consequence of those estimates. We also note that Theorem \ref{thm:3} is the only result that requires $\omega$ to be bounded from above,  and therefore is in a finite range. There are no such restrictions for all the other results Theorem \ref{thm:1},\ref{thm:2} and Corollary \ref{cor:Ber}. The main technical reason is that in the proof of Theorem \ref{thm:3} we used Chernoff–Hoeffding's inequality (see equation \eqref{eq:PEL}) for random variables in a finite range. This restriction for Theorem \ref{thm:3} might be relaxed by other methods, but we do not plan to dig into this in the current paper.

\section{Proof of the main results}\label{sec:proof}

\subsection{Preliminaries for the landscape function}\label{sec:pre}
The localization landscape function was introduced in \cite{FM} for a large class of differential operators under mild conditions. Most of the basic properties such as existence or positivity of the landscape function have been well established in the previous work \cite{FM,ADFJM-CPDE,ADFJM-SIAM}. 
 In this part, we summarize some of these properties for the very special case in which the one dimensional Schr\"odinger operator  
 \[-\Delta+V,\ \ V(x)\ge0,\]
 acts on functions defined on $\Omega=(0,L)$, with  Dirichlet boundary conditions. In general, $-\Delta+V$ should be understood in the weak sense, acting on $H^1_0(\Omega)$, either for the eigenvalue problem $(-\Delta+V)\phi=\lambda\phi$, or a boundary value problem $(-\Delta+V)u=1$. Recall $H^1(\Omega)$ is the usual Sobolev space given by the closure of $C^1(\Omega)$ functions, and $H^1_0(\Omega)$  is the closure of the subspace $C^1_0(\Omega)$ of continuously differentiable functions that are compactly supported in $\Omega$.  We also note that for problems on the 1-d interval $\Omega$, the weak solution of $(-\Delta+V)u=1$ coincides with the classical solution. In particular, we have the following results taken from \cite{FM,ADFJM-CPDE}. 
\begin{proposition}
Let $V$ be a bounded nonegative potential. Then there exists a unique  continuous  solution $u$ of $(-\Delta+V)u=1$, with  Dirichlet boundary conditions $u(0)=u(L)$. Moreover, $u$ is strictly positive on $(0,L)$.
\end{proposition}
Note that if $V$ is a nonegative potential, then all eigenvalues of $-\Delta+V$ on $H^1_0(\Omega)$ are strictly positive. Therefore, $-\Delta+V$ is invertible and $u$ can be solved directly as $u=(-\Delta+V)^{-1}1$. 
The positivity of the landscape function $u$ is actually a direct consequence of  the  following maximum principle which is a standard result for differential operators, can be found in e.g. \cite{GT,GP}.
\begin{lemma}\label{lem:pre1}
Suppose $V\ge0$, and $(-\Delta+V)f\ge 0$ weakly on $\Omega$. Then \[\min_{\Omega} f\ge \min\{0, f(0),f(L)\}.\] In particular, if $f(0)=f(L)\ge 0$, then $f\ge 0$ on $\Omega$. 
\end{lemma}
We will also frequently use this maximum principle to compare landscape functions of two Schr\"odinger operators with different potentials.
\begin{lemma}\label{lem:pre2}
Let $V_1\ge V_2\ge 0$ be two potentials on $\Omega$. Suppose $(-\Delta+V_1)u_1=1$, $(-\Delta+V_2)u_2=1$ on $\Omega$, and $0\le u_1(0)\le u_2(0)$, \ $0\le u_1(L)\le u_2(L)$. Then $u_1\le u_2$ on $\Omega$. 
\end{lemma}
\begin{proof}
It is enough to verify that $1=(-\Delta+V_1)u_1=(-\Delta+V_2)u_2$
implies
\[(-\Delta+V_2)(u_2-u_1)=(V_1-V_2)u_1\ge 0.\]
\end{proof}
On the other hand, the ground state eigenvalues of two Schr\"odinger operators can be compared using the min-max principle:
\begin{lemma}\label{lem:pre3}
Let $V_1\ge V_2\ge 0$ be two potentials on $\Omega$. Let $\lambda_1,\lambda_2$ be the smallest eigenvalue of the associated Schr\"odinger operators $-\Delta+V_1,-\Delta+V_2$, respectively. Then $\lambda_1 \ge \lambda_2$. 
\end{lemma}

\subsection{Proof of Theorem  \ref{thm:1}} \label{app:loc-pf} \label{sec:Bernoulli}



Let $H=-\Delta+kV_{\om}$ be as in \eqref{eq:H} and  with a nonegative $\om$-piecewise potential $V_{\om}$ as in \eqref{eq:V}. Assume that $\om$ satisfies \eqref{eq:sing-F}.
Throughout the rest of the paper, we denote by $\lambda=\lambda_1$ the ground state eigenvalue of $H$ for simplicity as long as there is no ambiguity. We also denote by $\max u=\max_{x\in \Omega} u$ when it is clear. 

The main work horse of this section is upper and lower bounds for the ground state eigenvalue  and the landscape function for a Bernoulli-piecewise potential (Lemma \ref{lem:Bernoulli-est} below). We will  prove Theorem \ref{thm:1} for general distributions using the estimates for the Bernoulli case. 
We write $\om\sim \Bern(p)$ for $p\in (0,1)$ if the random variable  $\om$ obeys the standard $\{0,1\}$ Bernoulli distribution
$
    \P(\om=0)=p$, $\P(\om=1)=1-p. 
$

We state the main estimates for a Bernoulli-piecewise potential $ V_{\om}^b$ in Lemma \ref{lem:Bernoulli-est} and prove Theorem \ref{thm:1} below. We delay the proof for Lemma \ref{lem:Bernoulli-est} until after the proof of Theorem \ref{thm:1}.

\begin{lemma} \label{lem:Bernoulli-est}
 Given $p>0$, let $\{\om_j\}_{j=1}^L\in \{0,1\}^L$ be any realization of a  Bernoulli trial given by $\Bern  (p)$. For any $b\in [0,\infty]$,   let $V_{\om}^b$ be a piecewise constant potential on $\R$ defined as
\begin{equation}\label{eq:Vb}
     V_{\om}^b(x)=
     b\om_j, \ \  x\in [j-1,j), \ j=1,\cdots,L .
 \end{equation}
Denote by $\lambda $ and $u$ the ground state eigenvalue, and the landscape function respectively for  $-\Delta+V_{\om}^b$ on $L^2(\Omega)$ with  Dirichlet boundary conditions.

Let $\ell_{\max}=\ell_{\max}(\om,L)$ be the longest length of an interval on which $V_{\om}^b=0$. Denote by $S=\max\{\sqrt b, 1\}$.  Then for any $b>0$ and $\ell_{\max}\ge 1$,
\begin{equation}\label{eqn:u-bounds} 
    \frac{\ell_{\max}^2}{8}\le   \max u \leq  \frac{3S\ell_{\max}}{b}+\frac{\ell_{\max}^2}{8}.
\end{equation}
Let $0 \le \nu<1, \gamma < 1$ be fixed. If $b\ell_{\max}^2 >\pi^2$ and  $b^{1-\nu}\ell_{\max}^\gamma >8\pi^2( 1+\sqrt b)$, then 
\begin{equation}
 \frac{\pi^2}{\ell_{\max}^2} \left(1- \frac{1}{b^{\nu/2} \ell_{\max}^{(1-\gamma)/2} } \right)^2  
   \leq  \lambda  
   \leq \frac{\pi^2}{\ell_{\max}^2}   .\label{eqn:lambda-bounds}
\end{equation}

In particular, if $b=\infty$, then 
\begin{equation}\label{eq:k-infty}
    \max u (x)=\frac{\ell_{\max}^2}{8}, \ \ {\rm and} \ \  \lambda    =\frac{\pi^2}{\ell_{\max}^2}. 
\end{equation}
\end{lemma}
\begin{remark}
The estimates of Lemma \ref{lem:Bernoulli-est} are deterministic and hold for any realization of the random potential. 
\end{remark}

\begin{proof}[Proof of Theorem \ref{thm:1}]
Let $V_{\om}$ be a $\om$-piecewise potential on $[0, L]$ be as in Theorem \ref{thm:1}. 
For any $\eps\ge 0$, let 
\begin{equation*}
    p_\eps=\P(\om \le \eps).
\end{equation*}
 Notice that $p_\eps\ge p_0\in (0,1)$ as assumed in \eqref{eq:sing-F}. An $\eps$-well 
is an interval $I\subset [0,L]$ with the longest length such that $V_{\om}(x)\le \eps$ for $x\in I$. We denote its length by $T_\eps$.   
 Let $L_{\max}$ be as in Theorem \ref{thm:1} for the longest length of a zero well for $V_{\om}$. We see that  $T_0=L_{\max}=\ell_{\max}$ as in Lemma \ref{lem:Bernoulli-est} and
 \begin{equation}\label{eq:Teps}
     \lim_{\eps\to0}T_\eps=L_{\max}.
 \end{equation} Moreover, let 
\begin{equation*}
    \eta^\eps_j=\begin{cases}
   0, \ \ & {\rm if}\ \ \om_j\le \eps \\
   1, \ \ & {\rm if} \ \ \om_j>\eps 
    \end{cases} . 
\end{equation*}
Then $\{\eta^\eps_j\}_{j=1}^L$ is a Bernoulli trail given by $\Bern(p_\eps)$, and $L_{\max}$ equals the longest length of a set of consecutive points $j\in\{1,\cdots,L\}$ such that $\eta^0_j=0$.  Given $p_0\in (0,1)$, \cite{Bishop} proved that
\begin{equation*}
    L_{\max}\to \infty \ \ {\rm as} \ \ L\to \infty
\end{equation*}
with probability one. 


 Let $I_{\max}=(a_1,b_1)\subset [0,L]$ be an interval of zero wells with maximum length $b_1-a_1=L_{\max}$.
Let $V^\infty_\om$ be defined as in \eqref{eq:Vb} for a trial so that $V^\infty_\om=0$ on $I_{\max}$ and $V^\infty_\om=\infty$ otherwise. The ground state eigenvalue $\lambda^\infty$ of $-\Delta+V^\infty_\om$ satisfies \eqref{eq:k-infty} with $\ell_{\max}=L_{\max}$. Since $V^\infty_\om \ge kV$ for any $k$, by the min-max principle and Lemma \ref{lem:pre3}, one has
\begin{equation}\label{eq:pre0}
    \lambda\le \lambda^\infty= \frac{\pi^2}{L_{\max}^2}.
\end{equation}
On the other hand, let $u^\infty$ be the landscape function of $-\Delta+V^\infty_\om$,  \eqref{eq:k-infty} implies $\max u^\infty(x)=  \dfrac{1}{8}L_{\max}^2.$  Note in this case, $u^\infty(x)= \dfrac{1}{2}(x-a_1)(b_1-x)$ on $I_{\max}$ which gives this maximum. 
Then by  $V^\infty_\om \ge kV$ and Lemma \ref{lem:pre2}, 
\begin{equation}\label{eq:tmp53}
 \max u(x)\ge \max u^\infty(x)=  \frac{L_{\max}^2}{8}. 
\end{equation}


Next, we estimate $\lambda$ from below and $u$ from above by constructing anther potential smaller than $V_{\om}$. Let $k$ satisfy \eqref{eq:k-beta}  and 
\begin{equation}\label{eq:eps-Lmax}
    \eps=\eps(L)={L_{\max}^{-\alpha/2}}.
\end{equation}
for $\alpha>0$ given as in \eqref{eq:k-beta}.  
For $k>0$, we define $V_{\om}^{k\eps}$ as in \eqref{eq:Vb} with $b=k\eps$. It is easy to verify that $V_{\om}^{k\eps}\le kV_{\om}$ for all $k>0$. Then by Lemma \ref{lem:pre2} and \ref{lem:pre3}, 
\begin{equation*}
    \lambda\ge \lambda^{k\eps}, \ \ {\rm and} \ \ \max u(x)\le \max u^{k\eps}(x),
\end{equation*}
where $\lambda^{k\eps}$ and $u^{k\eps}(x)$ are the first eigenvalue  and the landscape function of $-\Delta+V_{\om}^{k\eps}$ respectively. 
Suppose $b=k\eps\ge1$. Using the same the notation as Lemma \ref{lem:Bernoulli-est}, we see that $S=\max\{1,\sqrt {k\eps}\}=\sqrt{k\eps}$ and 
\begin{equation*}
    \max u^{k\eps}(x)\le  \frac{3T_\eps}{\sqrt{k\eps}}+\frac{T_\eps^2}{8}\le  \frac{T_\eps^2}{8}\left(1+\frac{24}{T_\eps}\right)\le \frac{T_\eps^2}{8}\left(1+\frac{24}{L_{\max}}\right).
\end{equation*}
Together with the upper bound of $\lambda$ in \eqref{eq:pre0}, one has 
\begin{equation}\label{eq:tempu1}
    \lambda \max u(x)\le \frac{\pi^2}{8}\frac{T_\eps^2}{L_{\max}^2}\left(1+\frac{24}{L_{\max}}\right).
\end{equation}
Suppose $b=k\eps<1$. Note that \eqref{eq:k-beta} and \eqref{eq:eps-Lmax} implies $k\eps>CL_{\max}^{\alpha/2-1}$.   We obtain from Lemma \ref{lem:Bernoulli-est}  
\begin{multline*}
 \max u^{k\eps}(x)\le  \frac{3T_\eps}{k\eps}+\frac{T_\eps^2}{8}\le  \frac{T_\eps^2}{8}\left(1+\frac{24}{k\eps T_\eps}\right)\\
 \le \frac{T_\eps^2}{8}\left(1+\frac{24}{CL_{\max}^{\alpha/2-1}L_{\max}}\right)\le \frac{T_\eps^2}{8}\left(1+\frac{24}{CL_{\max}^{\alpha/2 }}\right).   
\end{multline*}
  Similar to \eqref{eq:tempu1}, one has 
  \begin{equation}\label{eq:tempu2}
  \lambda \max u(x)\le \frac{\pi^2}{8}\frac{T_\eps^2}{L_{\max}^2}\left(1+\frac{24}{CL_{\max}^{\alpha/2}}\right).
  \end{equation} 

Since $L_{\max}\to \infty$ as $L\to \infty$ with probability one. Hence, equations \eqref{eq:Teps} and \eqref{eq:eps-Lmax} imply $\eps \to 0$ and $T_\eps/L_{\max}\to 1$ as $L\to \infty$ with probability one. Combing \eqref{eq:tempu1} and \eqref{eq:tempu2}, we obtain
\[\limsup_{L\to\infty}\lambda \max u(x) \le \frac{\pi^2}{8}.  \]

Now we turn to the lower bound of $\lambda \max u$. We apply \eqref{eqn:lambda-bounds} with $b=k\eps,\nu=0$ and an appropriate $0<\gamma<1$. If $\alpha\ge 2$, then we pick $\gamma=\frac{1}{2}$. We see that $b\eps>C$ and $(k\eps)\ell_{\max}^\gamma\gg 1+\sqrt{k\eps}$ holds for sufficiently large $\ell_{\max}=T_\eps$. Therefore, by \eqref{eqn:lambda-bounds}, 
\[
    \lambda\ge \lambda^{k\eps}\ge \frac{\pi^2}{\ell_{\max}^2} \left(1- \frac{1}{  \ell_{\max}^{1/4} } \right)^2  \ge \frac{\pi^2}{T_\eps^2} \left(1- \frac{1}{  L_{\max}^{1/4} } \right)^2 .
\]
Together with \eqref{eq:tmp53}, one has
\begin{equation} \label{eq:tmp55}
 \lambda   \max u(x)\ge \frac{\pi^2}{8} \frac{L_{\max}^2}{T_\eps^2} \left(1- \frac{1}{  L_{\max}^{1/4} } \right)^2.
\end{equation}
If $\alpha<2$, then we pick $\gamma=(\alpha/2+1)/2\in (0,1)$. In this case, $(k\eps)\ell_{\max}^\gamma>k L_{\max}^{\gamma-\alpha/2}\gtrsim 1+\sqrt {k}> 1+\sqrt {k\eps}$ for sufficiently large $L_{\max}$. Similar to \eqref{eq:tmp55}, we obtain
\begin{equation} \label{eq:tmp56}
 \lambda   \max u(x)\ge \frac{\pi^2}{8} \frac{L_{\max}^2}{T_\eps^2} \left(1- \frac{1}{  L_{\max}^{(1-\gamma)/2} } \right)^2.
\end{equation}
Combing \eqref{eq:tmp55} and \eqref{eq:tmp56},
\[
 \liminf_{L\to\infty}\lambda   \max u(x)\ge \frac{\pi^2}{8} .\]
It follows that so long as
\begin{equation}\label{eq:pf5}
 kL_{\rm max}^{1-\alpha}>C   
\end{equation}
 for some $C>0,\alpha>0$, \eqref{eq:main-pi28} is proved.
 
 The case where $k$ is fixed is a direct consequence of \eqref{eq:pf5}. Since $
    L_{\max}(\om,L)\to \infty \ \ {\rm as} \ \ L\to \infty$
with probability one. Therefore, \eqref{eq:pf5} holds with $\alpha=1/2,C=1$, and any fixed $k>0$. The case in Theorem \ref{thm:1} where $k$ is fixed follows directly from \eqref{eq:pf5} and \eqref{eq:main-pi28}.
\end{proof}

Now, we complete the proof of the main work horse Lemma \ref{lem:Bernoulli-est}.
\begin{proof}[Proof of Lemma \ref{lem:Bernoulli-est}]

Throughout the proof, 
we enumerate all the wells of $V_{\om}^b$. Let $\{I_i\}_{i=1,...,m}$ denote the collection of disjoint intervals of maximum length on which $V_{\om}^b = 0$. Moreover, we order the set $\{I_i\}$ such that $I_i$ is to the right of $I_j$ if $i > j$. For each $I_i$, let $l_i, c_i, r_i$, and $L_i$ denote the left end point, center, right end point, and length, respectively. Finally, let $I_{\max}$ denote an interval with the  longest length on which $V_{\om}^b = 0$ with length $\ell_{\max}=\max_i L_i$.

\textbf{1. Lower bound for $u$.}

 Let $I_{\max}=(a_1,b_1)$ be given as above. Let $\wt u(x)=\dfrac{1}{2}(x-a_1)(b_1-x)$. Then $\wt u > 0 $ on $I_{\max}$, and $\wt u\le0$ otherwise. Moreover, $-\Delta \wt u=1$ for all $x\in [0,L]$. Therefore, 
\[ (-\Delta+V_{\om}^b) \wt u(x)\begin{cases}
=1,\qquad \ & x\in I_{\max}\\
\le 1,    & {\rm otherwise},
\end{cases}\]
which implies $(-\Delta+V_{\om}^b)(u-\wt u)(x)\ge 0$. Notice that $\wt u(0)\le 0 =u(0)$ and $\wt u(L)\le 0 =u(L)$. The maximum principle Lemma \ref{lem:pre1} gives $u(x)\ge \wt u(x) $ on $[0,L]$. In particular, 
\[
  \max u   \geq \frac{\ell_{\max}^2}{8}.
\]
This proves the lower bound in \eqref{eqn:u-bounds}.

\textbf{2. Upper bound for $u$.}
Let $I_i$, $l_i, c_i, r_i$,  $L_i$, and $\ell_{\max}$ be given as before.  We use the notation $x > I_i$ ($x < I_i$) to mean $x > y$ ($x < y$) for all $y \in I_i$. For $S=\max\{\sqrt b,1\}$, define
\begin{equation}
  \sigma_i(x) = 
  \begin{cases}
  -\frac{1}{2}(x-c_i)^2 + \frac{L_i^2}{8}+\frac{L_i }{4S}  & x \in I_i \\
                                              \frac{ S}{4}L_i (x-r_i-\frac{1}{S})^2 & x\in (r_i,r_i+\frac{1}{S}) \\
                        \frac{S}{4} L_i (x-l_i+\frac{1}{S})^2 & x\in (l_i-\frac{1}{S},l_i)\\
                      0 & {\rm otherwise}
  \end{cases} . 
  \label{eqn:sigma-i-def}
\end{equation}
We note that $\sigma_i$ is $C^1$. Moreover, it satisfies the following differential equation
\[  -\Delta \sigma_i(x) = 
  \begin{cases}
  1  & x \in I_i \\
                                             -\frac{ S}{2}L_i & x\in (l_i-\frac{1}{S},l_i)\cup (r_i,r_i+\frac{1}{S}) \\
                      0 & {\rm otherwise}
  \end{cases} . 
\]
Hence, 
\begin{equation}
  (-\Delta+ V_{\om}^b)\sigma_i  \begin{cases}
 = 1  &  x \in I_i \\
                                           \ge -\frac{ S}{2}L_i & x\in (l_i-\frac{1}{S},l_i)\cup (r_i,r_i+\frac{1}{S}) \\
                      =0 & {\rm otherwise}
  \end{cases} .  \label{eqn:sigma-i-properties}  
\end{equation}
We construct $\wt u$ via
\begin{equation*}
  \wt u = \frac{1+S\ell_{\max}}{b}  + \sum_i \sigma_i \, . \label{eqn:sigma-0-def}
\end{equation*}
Note that $\sigma_i,\sigma_j$ are pairwise disjoint for $|i-j|\ge 2$ since $S\ge 1$.  
By  \eqref{eqn:sigma-i-properties}, we see that  $\wt u$ is a sup-solution:
\[(-\Delta+ V_{\om}^b)  \wt u \ge 1=(-\Delta+ V_{\om}^b)  u, \]
and $\wt u(0),\wt u(L)\ge 0$ by the construction. 
 It follows by the maximum principle Lemma \ref{lem:pre2} of $-\Delta+V_{\om}^ b$ that $\wt u\geq u$ and hence
\begin{align*}
  \max u  \leq \max \wt u=& \frac{1+S\ell_{\max}}{b} + \max_i \max_{x \in I_i} \sigma_i \\
  \le &  \frac{2S\ell_{\max}}{b}+\frac{\ell_{\max}^2}{8}+\frac{\ell_{\max} }{4S}\le  \frac{3S\ell_{\max}}{b}+\frac{\ell_{\max}^2}{8},
\end{align*}
which is \eqref{eqn:u-bounds}.

\textbf{3. Upper bound on $\lambda$.}
Let $I_{\max}=(a_1,b_1)$ be given as above, i.e., an interval with the longest length $\ell_{\max}$ such that $V^b_{\om}(x)=0$. The min-max principle shows
\begin{multline}
    \lambda=\min_{\phi\in{H_0^1(\Omega)}}\frac{\ipc{\phi}{(-\Delta+V^b_{\om})\phi}}{\ipc{\phi}{\phi}}\\ \le \min_{\phi\in {H_0^1(I_{\max})}}\frac{\ipc{\phi}{(-\Delta+V^b_{\om})\phi}}{\ipc{\phi}{\phi}}
    = \min_{\phi\in {H_0^1(I_{\max})}}\frac{\ipc{\phi}{-\Delta\phi}}{\ipc{\phi}{\phi}}= \frac{\pi^2}{\ell_{\max}^2}. \label{eqn:lambda-upper-bound}
\end{multline}


\textbf{4. Lower bound for $\lambda$.}
In this part, we prove the lower bound \eqref{eqn:lambda-bounds}. 
The main idea of the claim follows \cite{Bishop}, which proved a similar claim for the discrete case. In a nutshell, we locate all the wells in which the ground state concentrates and compute the energy lower bound on these wells. Retain the definition of $V_{\om}^b$ and $\ell_{\max}$ in Lemma \ref{lem:Bernoulli-est}. First, let $B$ denotes the set on which $V_{\om}^b = b$. Then we have the following result.

\begin{lemma} \label{lem:psi-small-on-K}
  Let $\psi$ denote a normalized ground state of $-\Delta + V_{\om}^b$. Then
  \begin{equation}
    \|\psi\|_{L^2(B)} \leq \frac{\pi^2}{b\ell_{\rm max}^2}. \label{eqn:psi-small-on-K}
  \end{equation}
\end{lemma}
\begin{proof}[Proof of Lemma \ref{lem:psi-small-on-K}]
  We recall from the upper bound estimate \eqref{eqn:lambda-upper-bound} that
\[
    \frac{\pi^2}{\ell_{\rm max}^2} \geq \lan \psi, (-\Delta +V_{\om}^b) \psi \ran \geq \|V_{\om}^b\psi\|_{L^2(B)}^2 = b\|\psi\|_{L^2(B)}^2.
\]
  After dividing by $b$ in the equation above, \eqref{eqn:psi-small-on-K} is established.
\end{proof}

Let $I_i,l_i,r_i,L_i$ be given as before (see e.g.  \eqref{eqn:sigma-i-def}). 
Lemma \ref{lem:psi-small-on-K} implies that $\psi$ is concentrated on $B^c=\cup_iI_i$, the set of wells, as we would expect. For each well $I_i$, let us denote
\begin{equation}
  m_i := \|\psi\|_{L^2(I_i)} \label{eqn:m-def}
\end{equation}
and, since $\psi \in H^2(\R)$ is continuous, we further define
\begin{equation}
  \delta^L_i :=  \psi(l_i)/m_i,\ \   \quad  \delta^R_i :=  \psi(r_i)/m_i. \label{eqn:delta-def}
\end{equation}
Finally, we define the following notion of concentration.

\begin{definition} \label{def:heavy-light-def}
  Let $\psi$ be a ground state of $-\Delta + V_{\om}^b$. Let $0\le \nu,\gamma <1$ be fixed. A well, $I_i$, is called heavy if
  \begin{align}
    \max(\delta_i^L, \delta_i^R)^2 \leq L_i^{-1} b^{-\nu}\ell_{\rm max}^{\gamma-1} \label{eqn:heavy-def}.
  \end{align}
  Otherwise, it is called light.
\end{definition}

Let $N$ denote the union of all light intervals.

\begin{lemma} \label{lem:psi-large-on-M}
  Let $\psi$ be a normalized ground state of $-\Delta+V_{\om}^b$. 
  If $0<\gamma<1$ and $b\ell_{\max}^2 >\pi^2$, then
  \begin{equation}\label{eqn:psi-large-on-M}  
     \|\psi\|_{L^2(N)}^2 \leq    \frac{8\pi^2(1+\sqrt b)}{b^{1-\nu}\, \ell_{\rm max}^{ \gamma}}  . 
  \end{equation}
  If, in addition, ${b^{1-\nu}\, \ell_{\rm max}^{ \gamma}}>8 \pi^2( 1+\sqrt b )$, then we have at least one heavy well.
\end{lemma}
\begin{proof}

By definition of lightness (See definition \ref{def:heavy-light-def}),
\begin{multline}
      \|\psi\|_{L^2(N)}^2  = \sum_{\text{light $I_i$}} m_i^2                        \leq \sum_{\text{light $I_i$}} m_i^2 \max(\delta_i^L, \delta_i^R)^2  \, L_i \,  b^{\nu} \,  \ell_{\rm max}^{1-\gamma}  \\
                         \leq  b^{\nu}\, \ell_{\rm max}^{2-\gamma} \sum_{\text{light $I_i$}} \max(\psi(l_i), \psi(r_i))^2. \label{eqn:L-2-norm-on-N-upper-bound}
\end{multline}
  We claim that
  \begin{equation}
    \max(\psi(l_i), \psi(r_i))^2 \leq 4(1+\sqrt{b})\max(n_{i-1}, n_{i+1})^2  \label{eqn:BC-pt-upper-bound}
  \end{equation}
  where $n_{i\pm 1}$ is the $L^2$ norm of $\psi$ on the 2 neighboring walls of $I_i$ (i.e. left and right intervals with $V_{\om}^b=b$). We delay the proof of this claim until the next paragraph and complete the proof of Lemma \ref{lem:psi-large-on-M} first. The claim and \eqref{eqn:L-2-norm-on-N-upper-bound} show that
  \[  \|\psi\|_{L^2(N)}^2 \leq 8(1+\sqrt{b})\,  \ell_{\rm max}^{2-\gamma}\, \|\psi\|_{L^2(B)}^2, \]
  where we recall that $B$ is the union of all the walls (i.e sets where $V_{\om}^b=b$). Invoking Lemma \ref{lem:psi-small-on-K}, we see that
  \eqref{eqn:psi-large-on-M} is proved. The remainder of the paragraphs proves claim \eqref{eqn:BC-pt-upper-bound}.

  In fact, we prove claim \eqref{eqn:BC-pt-upper-bound} for the following more general setting. Let $I = [s,t]$ denote an interval on which $V_{\om}^b=b$. We show that the end points of $I$ satisfy
  \[
    \max\big(\psi(s), \psi(t)\big)^2 \leq 4 (1+\sqrt{b})\|\psi\|_{L^2(I)}^2.\]
  Without loss of generality, we assume that $s=0$. Let
  \begin{equation}\label{eqn:kappa-def}
         \kappa :=  t\sqrt{b-\lambda},  \ \ 
    l :=  \psi(0),\ \  \text{ and}  \ \ 
    r :=  \psi(t).
  \end{equation}
  We make a brief remark here that $\lambda < b$  since $\lambda < \frac{\pi^2}{L_{\rm max}^2}$ by the upper bound estimate for $\lambda$ (see \eqref{eqn:lambda-upper-bound}), provided $b\ell_{\max}^2 > \pi^2 $.  
  It is elementary to check that
  \[ \psi(x) = Ae^{\sqrt{b-\lambda}\, x} + Be^{-\sqrt{b-\lambda}\, x}\ \ {\rm for}\ \ x\in I, \]
 where the coefficients $A$ and $B$ are given by
 \begin{equation}\label{eqn:AB-eqn}
      A  = \frac{r - l e^{-\kappa}}{e^{\kappa} - e^{-\kappa}}, \ \     B  = \frac{le^{\kappa} - r }{e^{\kappa} - e^{-\kappa}} , 
 \end{equation}
  where $\kappa, l, r$ are defined in \eqref{eqn:kappa-def} , respectively.
  Using \eqref{eqn:AB-eqn}, we see that the $L^2$-norm of $\psi$ on this wall is
  \begin{align}
    \|\psi\|_{L^2(I)}^2  =& \frac{A^2}{2\sqrt{b-\lambda}}(e^{2\kappa}-1) + \frac{B^2}{2\sqrt{b-\lambda}}(1-e^{-2\kappa}) + 2ABt  \notag \\
    =& \frac{2\cosh(\kappa)(r^2 + l^2) - 4lr}{4\sqrt{b-\lambda}\sinh(\kappa)}- \frac{t(r^2 + l^2 -2 rl \cosh(\kappa))}{2\sinh^2(\kappa)} \notag \\
    =& \frac{l^2+r^2}{2\sqrt{b-\lambda}}\left( \frac{\cosh(\kappa)\sinh(\kappa)-\kappa}{\sinh^2(\kappa)} \right)  +  \frac{rl}{\sqrt{b-\lambda}}\left(\frac{\kappa\cosh(\kappa)-\sinh(\kappa)}{\sinh^2(\kappa)}\right). \label{eqn:L-2-norm-of-psi-on-I-1}
  \end{align}
  Since $V_{\om} \geq 0$, positivity of $-\Delta+V_{\om}$ implies that $\psi \geq 0$. That is, $l,r \geq 0$. Moreover, since $x\cosh(x) \geq \sinh(x)$ for $x \geq 0$, equation \eqref{eqn:L-2-norm-of-psi-on-I-1} shows that
  \begin{equation}
       \|\psi\|_{L^2(I)}^2 \geq  \frac{l^2+r^2}{2\sqrt{b-\lambda}} \left( \frac{\cosh(\kappa)\sinh(\kappa)-\kappa}{\sinh^2(\kappa)} \right) 
    \geq  \frac{l^2+r^2}{4\sqrt{b-\lambda}} \left( \frac{\kappa}{1+\kappa} \right), \label{eqn:L-2-norm-of-psi-on-I-2}
  \end{equation}
  where in the last line we have used 
  the elementary fact that 
  \[\frac{\sinh(x)\cosh(x)-x}{\sinh^2(x)} \geq \frac{x}{2(1+x)} \ \ \ {\rm for} \ \ x \geq 0.\] Since the function $\frac{x}{1+x}$ is increasing and the definition $\kappa = t\sqrt{b-\lambda}\ge \sqrt {b-\lambda}$, we conclude from \eqref{eqn:L-2-norm-of-psi-on-I-2} that
\[
    \|\psi\|_{L^2(I)}^2  \geq \frac{l^2+r^2}{4\sqrt{b-\lambda}} \left( \frac{\sqrt{b-\lambda}}{1+\sqrt{b-\lambda}} \right) \geq \frac{\max(l,r)^2}{4(1+\sqrt{b})}
\]
  and the claim \eqref{eqn:BC-pt-upper-bound} is proved. This concludes the proof of Lemma \ref{lem:psi-large-on-M}.

\end{proof}

Now, we proceed with the final proof of the lower bound for $\lambda$.  Let $\psi$ denote a normalized ground state of $-\Delta + V_{\om}^b$ associated with $\lambda$. We consider an arbitrary well $J \subset M$. Again, without loss of generality, we assume that $J = [0, T]$, for some  $T\in (0,\ell_{\max}]$.
Elementary calculus shows that on $J \subset \R$,
\[
  \psi(x) = c m \sin\left( \sqrt \lambda\,  x + \theta \right) \]
for some constants $c,   0 \leq \theta \leq \frac{\pi}{2}$ and where recall that $m = \|\psi\|_{L^2([0,T])}$. We also define $s$ through the relation $s\pi=  {T} \sqrt\lambda$.

We would like to estimate $s$ from below. Using definitions \eqref{eqn:m-def} and \eqref{eqn:delta-def},  boundary conditions require
\[  m\delta^L  := \psi(0) = cm\sin(\theta) , \ \ \  {\rm and}\ \ \ 
  m\delta^R  :=  \psi(T) = cm\sin\left(s\pi  + \theta \right). \]
Solving for $s\ge 1/2$, we obtain
\begin{equation}
  s = 1 -\frac{1}{\pi} \left(\arcsin(\delta^L  c^{-1}) + \arcsin(\delta^R  c^{-1})\right). \label{eqn:s-formula}
\end{equation}
Note that the left boundary condition $\delta^L =c\sin(\theta)$ was solved on $[0,\pi/2]$ and  the right boundary condition  $\delta^R  = c\sin\left(s\pi  + \theta \right)$ was solved on $[\pi/2,3\pi/2]$. 
Normalization requires
\begin{equation}
  m^2 = \|\psi\|_{L^2(W)}^2 =  c^2 m^2 \int_0^{T} \sin^2(\sqrt \lambda x + \theta) dx. \label{eqn:c-est-by-normalization}
\end{equation}
Since $\sin^2 \leq 1$, it follows by \eqref{eqn:c-est-by-normalization} that
\begin{equation}
  c^{-1} \leq \sqrt T. \label{eqn:c-est-result}
\end{equation}
Combining \eqref{eqn:s-formula} and \eqref{eqn:c-est-result}, we deduce that
\[
  s \geq 1 - \frac{2}{\pi} \max\left (\arcsin(\delta^L\sqrt{T}), \arcsin(\delta^R\sqrt{T}) \right)\ge  1 - \frac{2}{\pi} \arcsin\left(\max(\delta^L, \delta^R) \sqrt{T}\right).
\]
Suppose $J=[0,T]$ is a heavy well defined as in \eqref{eqn:heavy-def}. It follows by the definition of heaviness that
\begin{equation*}
     s \geq  1  - \frac{2}{\pi} \arcsin\left( b^{-\nu/2}\,  \ell_{\rm max}^{(\gamma-1)/2}\right)  
  \geq  1 -  b^{-\nu/2}\, \ell_{\max}^{(\gamma-1)/2 }    
\end{equation*}
since $\arcsin(x)$ is bounded by $\frac{\pi}{2}x$ on $[0,1]$.

Finally, we estimate $\lambda$ from below by $s$.
\begin{equation*}
    \lambda=\frac{\pi^2}{T^2}s^2\ge \frac{\pi^2}{\ell_{\max}^2}\, \left(1 -  \frac{1}{b^{\nu/2}\, \ell_{\max}^{(1-\gamma)/2 } }  \right)^2. 
\end{equation*}

This proves the lower bound in \eqref{eqn:lambda-bounds}.
\end{proof}

\subsection{Semi-classical regime: proof of Theorem \ref{thm:2}} \label{sec:semi-classical}

\begin{proof}[Proof of Theorem \ref{thm:2}]
Let $H=-\Delta+kV_{\om}$ be as in \eqref{eq:H} with a  $\om$-piecewise potential $V_{\om}$ satisfying $\om\ge0$. Let $p=\P(\om = 0)$. Let $\lambda$ be the first eigenvalue and let $u$ be the landscape function of $H$.

\textbf{Case 1: $p>0$ and $\inf V_{\om} = 0$.} \\
This case occurs with probability $1-(1-p)^{L}$. Use the same notation as in Lemma \ref{lem:Bernoulli-est}. Let $I_{\max}$ be a longest interval on which $V_{\om}=0$ and denote by $L_{\max}$ its length.  Fix $L$, let 
\begin{equation}\label{eq:Aa}
  A=\max_{V_{\om}(x)>0}V_{\om}(x),\ \ {\rm and} \ \    a=\min_{V_{\om}(x)>0}V_{\om}(x)>0.    
\end{equation}
For $b>0$,  let $V_{\om}^b$ be a piecewise constant potential as in \eqref{eq:Vb}. Clearly, 
\[
V_{\om}^{ka} \le   kV_{\om}\le V_{\om}^{kA}. 
\]
By the maximum principle, 
\[
  \max  u^{kA}\le  \max u \le \max u ^{ka},
\]
where $u,u^{kA},u^{ka}$ are the landscape functions associated to the potentials $kV_{\om},V_{\om}^{kA},V_{\om}^{ka}$ respectively. Applying  \eqref{eqn:u-bounds} of Lemma \ref{lem:Bernoulli-est} to $u^{kA},u^{ka}$ with $b>1$ gives
\[
 \frac{L_{\max}^2}{8}  \le     \max u(x)\le   \frac{L_{\max}^2}{8}+ \frac{3}{ \sqrt{ka}}L_{\max}.
\]

Taking the limit as $k\to \infty$, we have 
\[
\lim_{k \rightarrow \infty}\, \max u(x)= \frac{L_{\max}^2}{8} . 
\]
 Similarly, 
\[
   \lambda^{ka} \le   \lambda\le \lambda^{kA},
\]
where $\lambda,\lambda^{kA}\lambda^{ka}$ are the first eigenvalues associated to the potentials $kV_{\om},V_{\om}^{kA},V_{\om}^{ka}$ respectively. 
 Applying  \eqref{eqn:lambda-bounds} of Lemma \ref{lem:Bernoulli-est} to $\lambda^{kA}$ implies that 
 $\lambda^{kA}\le \frac{\pi^2}{L_{\max}^2}. $
 Then apply  \eqref{eqn:lambda-bounds}  to $\lambda^{ka}$ with $\nu=1/4$ and $\gamma=0$, one gets
 \[ \lambda^{ka}\ge  \frac{\pi^2}{L_{\max}^2} \left(1- \frac{1}{(ka)^{1/8} L_{\max}^{1/2} } \right)^2  \] Hence,  
 \[
     \lim_{k \rightarrow \infty}\,    \lambda =\frac{\pi^2}{L_{\max}^2},
\]
and
\[    \lim_{k \rightarrow \infty}\,  \lambda \max  u  =\frac{\pi^2}{8} \]
 with probability $1-(1-p)^{L}$. 
This completes the proof for \eqref{eq:semi-1}.

\textbf{Case 2: $\inf V_{\om}  > 0$.}

This case occurs with probability $(1-p)^{L}$. Recall that for any non-negative potential, we have the landscape uncertainty principle that for any $\varphi$
\[
    \ipc{\varphi}{H  \varphi}\ge \ipc{\varphi}{\frac{1}{u}\varphi}\ge \frac{1}{\max u }\ipc{\varphi}{ \varphi}.
\]
This implies the lower bound $    \lambda\max u \ge 1
$
for any $k>0,L\ge 1$.

It is enough to obtain a upper bound for the above product in large $k$ limit. 
We obtain a bound for $\lambda$ first. Let $a=\min_{[0,L]}V_{\om}>0$ be as in \eqref{eq:Aa} so that $kV_{\om}\ge ka$. 
Let $\lambda$ be the first eigenvalue of $-\Delta+kV_{\om}$ and  $\wt \lambda$ be the first eigenvalue of $-\Delta+kV_{\om}-ka$. Clearly, $\lambda=ka+\wt \lambda $. One can show that $0\le \wt \lambda \le \pi^2$. Therefore, $
ka \le   \lambda\le ka+\pi^2. $

Next, we estimate $u$. 
Consider a constant function $u^{\sup}(x) \equiv {1}/({ka})$ for $x\in [0,L].$  Clearly, $(-\Delta+kV_{\om})u^{\sup}\ge 1$ on $(0,L)$ and $u^{\sup}(0),u^{\sup}(L)>0$ at the boundary points. The maximum principle Lemma \ref{lem:pre1} implies $     u \le 1/(ka).$
Note that $a>0$ is independent of $k$. Therefore, 
\[
    \limsup_{k\to \infty}\, \left(\lambda  \max u \right)\le  \limsup_{k\to \infty}  \frac{ka+\pi^2}{ka}=1,
\]
which completes the proof for \eqref{eq:semi-2}.
\end{proof}

\subsection{Proof of Theorem \ref{thm:3}} \label{sec:general}
 We define a characteristic quantity as
\begin{equation*}
   \gamma_{\rm c} := k L^2\, \E(\om). \label{eq:gamma-c-def}
\end{equation*}
We consider the regime where $1+\gamma_{\rm c}\ll L^{\beta}$ and prove the following theorem.  

\begin{theorem} \label{thm:33}
Let $H=-\Delta+kV_{\om}$ be as in \eqref{eq:H} with $\om$-piecewise potential $V_{\om}$ as in \eqref{eq:V}. Let $\lambda$ and $u$ be the ground state eigenvalue and the landscape function  of $H$ on $[0,L]$ with  Dirichlet boundary conditions, respectively.  Suppose $\omega$ is nonegative and bounded from above.
Assume that $k$ is chosen so that 
\begin{equation*}
    1+  \gamma_{\rm c}   <C L^{\beta} 
\end{equation*}
for some $C>0$ and $0 < \beta < \frac{1}{4}$ as $L \rightarrow \infty$. Then
\begin{equation}
    \lim_{L \rightarrow \infty} \frac{\lambda/L^2}{\pi^2 + \gamma_{\rm c}} = 1, \label{eqn:delocalized-main-result-1}
\end{equation}
and 
\begin{equation}
\lim_{L \rightarrow \infty} \frac{L^2\max u }{\frac{1}{\gamma_{\rm c}}\left(1-\frac{1}{\cosh(\sqrt{\gamma_{\rm c}}/2)} \right)} = 1. \label{eqn:delocalized-main-result-2}
\end{equation}
with probability one. 

\end{theorem}

We first use Theorem \ref{thm:33} to complete 
\begin{proof}[Proof of Theorem \ref{thm:3}]
If we multiply \eqref{eqn:delocalized-main-result-1} and \eqref{eqn:delocalized-main-result-2}, we arrive at
\begin{equation}\label{eq:ratio-deloc}
  \lim_{L\to \infty}   \frac{\lambda \max u }{R(\gamma_{\rm c})}  = 1
\end{equation}
with probability one, where 
\begin{equation*}
    R(\gamma_{\rm c})=\frac{\pi^2 + \gamma_{\rm c}}{\gamma_{\rm c}}\left(1-\frac{1}{\cosh(\sqrt{\gamma_{\rm c}}/2)} \right).
\end{equation*}
We note that $R(\gamma_{\rm c})$ ranges (continuously) from $1$ to $\pi^2/8$ as $\gamma_{\rm c}$ ranges from $\infty$ to $0$. Given $r\in(0,\pi^2/8)$, we solve for $\gamma_{\rm c}^*$ such that $R(\gamma_{\rm c}^*)=r$ and let $k=\frac{\gamma_{\rm c}^*}{L^2\E(\om)}$. Then \eqref{eq:ratio-deloc} implies that 
\[\lim_{L\to \infty}   \frac{\lambda \max u }{R(\gamma_{\rm c}^*)}  = 1,\]
which proves \eqref{eq:ratio-any} for any $r\in(0,\pi^2/8)$. Notice that if $r=\pi^2/8$, then $\gamma_{\rm c}^*=R^{-1}(\pi^2/8)=0$. As long as one picks $k=k(L)$ as that $\gamma_{\rm c}=kL^2\E(\om)\to 0$ as $L\to \infty$, then \eqref{eq:ratio-deloc} implies that 
\[1=\lim_{L\to \infty}   \frac{\lambda \max u}{R(\gamma_{\rm c})}=\lim_{L\to \infty}   \frac{\lambda \max u}{R(0)} ,\]
which is \eqref{eq:weak2}. The argument for the case $r=1$ and \eqref{eq:weak1} is exactly the same.  
\end{proof} 

The rest of the section is devoted to the proof of Theorem \ref{thm:33}. 
Recall that $L$ is the length of the domain $[0, L]$ on which we study the eigenvalue problem and the landscape function of $-\Delta + kV_{\om}$. We begin by performing a rescaling to facilitate a homogenization effort performed below. Let
\[
    (Uf)(x) = \sqrt{L}f(L x).
\]
We note that
\begin{equation*}
    L^2UHU^* = -\Delta + kL^2 V_{\om,L} =: H_L, \label{eqn:H-L-def}
\end{equation*}
where $V_{\om,L}(x) = V_{\om}(Lx)$. Note that $\lambda_L$ is the ground state eigenvalue of $H_L$ if and only if $\lambda_L/L^2$ is the ground state eigenvalue of $H$. Similarly, if $u_L$ solves
\[
    H_L u_L = 1,
\]
then,
\[
    u =L^2\sqrt L\,  (U^*u_L)(x)= L^2\, u_L(x/L).
\]
In particular,
\begin{equation*}
    \lambda \max_{x\in[0,L]} u = \lambda_L \max _{x\in[0,1]} u_L. \label{eqn:lambda-u-lambda_L-u-L-equiv}
\end{equation*}
Consequently, we estimate $\lambda_L$ and $u_L$.

We homogenize $V_{\om,L}$ via by taking its average. Let 
$\gamma_c=kL^2\E(\omega)$ be the characteristic scale as in \eqref{eq:gamma-c-def}. Respectively, let $\lambda_{\rm c}$ and $u_{\rm c}$ denote the ground state eigenvalue  and the landscape function for
\begin{equation}
    H_{\rm c} := -\Delta + \gamma_{\rm c} \label{eqn:H_mu-def}
\end{equation}
on the domain $[0, 1]$ with  Dirichlet boundary conditions. We will show that the $\lambda_{\rm c}$ and $\lambda_L$, and $u_{\rm c}$ and $u_L$ are sufficiently close in subsections \ref{subsec:weak-p-u-estimate} and \ref{subsec:weak-p-lambd-estimate}, respectively. We conclude the proof for Theorem \ref{thm:33} after these two subsections.

\subsubsection{Estimate for the landscape function.} \label{subsec:weak-p-u-estimate}

The following Lemma is the main result of this subsection. Let $\gamma_{\rm c}$ be given in \eqref{eq:gamma-c-def} and $L$ denote the length of the underlying domain $[0,L]$.
\begin{lemma} \label{lem:u_L-main-lemma}
Assume that

\begin{equation*}
    1+ \gamma_{\rm c} <C L^{\beta}
\end{equation*}
for some constant $C>0$ and $\beta < \frac{1}{4}$. 
There is a constant $C_1$ only depending on the range of $\omega$ and a constant $C_2>0$ only depending on $\E(\omega)$ such that 
\begin{equation}\label{eq:uc-uL-probH}
 \Big | \frac{\max_{x\in[0,1]} u_L}{\max_{x\in[0,1]} u_{\rm c}}-1 \Big| \le C_2\, L^{-(1/4-\beta)}  
\end{equation}
with probability $1-e^{-C_1\, L^{(1/2-2\beta)^2}}$ as $L \rightarrow \infty$.

As a direct consequence of the Borel–Cantelli lemma, 
\begin{equation}\label{eq:uc-uL-prob1}
   \lim_{L\to\infty} \frac{\max_{x\in[0,1]} u_L}{\frac{1}{\gamma_{\rm c}}\left(1-\frac{1}{\cosh(\sqrt{\gamma_{\rm c}}/2)} \right)}  =1  
\end{equation}
with probability one. 
\end{lemma}

\begin{lemma} \label{lem:u-c-lemma}
The landscape function $u_{\rm c}$ for $H_{\rm c}$ (see \eqref{eqn:H_mu-def}) is
\[
    u_{\rm c}(x) = \frac{1}{\gamma_{\rm c}}\left(1-\frac{\cosh(\sqrt{\gamma_{\rm c}}(x-1/2)}{\cosh(\sqrt{\gamma_{\rm c}}/2)} \right) \ \ {\rm for}\ \ x\in[0,1].
\]
Moreover, 
\begin{equation}\label{eq:uc-sup}
    \max_{x\in[0,1]} u_{\rm c} =\frac{1}{\gamma_{\rm c}}\left(1-\frac{1}{\cosh(\sqrt{\gamma_{\rm c}}/2)} \right)
\end{equation}
and 
\[\|u_{\rm c}\|_{H^1}\le \frac{1+\sqrt{\gamma_{\rm c}}}{\gamma_{\rm c}}.\]
\end{lemma}
The proof of Lemma \ref{lem:u-c-lemma} is elementary and is omitted. We proceed to prove Lemma \ref{lem:u_L-main-lemma}.

\begin{proof}[Proof of Lemma \ref{lem:u_L-main-lemma}]
To extract leading order behavior, we decompose
\begin{equation}
    H_L = H_{\rm c} + \wt V_{\om}, \label{eqn:H-T-decomp}
\end{equation}
where ${\wt V_{\om}}=kL^2\big(V_{\om,L}-\E(\omega)\big)$. By repeated application of the identity 
\[
    (H_{\rm c} + {\wt V_{\om}})^{-1} = H_{\rm c}^{-1} - (H_{\rm c} + {\wt V_{\om}})^{-1}{\wt V_{\om}}H_{\rm c}^{-1},
\]
we see that
\begin{equation}
    u_L := H_L^{-1} 1 = (H_{\rm c} + {\wt V_{\om}})^{-1} 1 \\
    = u_{\rm c} + \sum_{n \geq 1} (-1)^n(H_{\rm c}^{-1} {\wt V_{\om}})^{n} u_{\rm c}, \label{eqn:H_L-expand}
\end{equation}
whenever the serious converges. Using this series expansion, we show that the following Lemma holds. Let 
\begin{equation}
    F(x) := \int_0^x {\wt V_{\om}}(y) dy \label{eqn:F-anti-d-def} . 
\end{equation}

\begin{lemma} \label{lem:u-T-first-order}
Assume that $\|F\|_2 \ll (1+\sqrt{{\gamma_{\rm c}}})^{-1}$,
then 
\begin{equation*}
    \|u_L - u_{\rm c}\|_{H^1} \lesssim (1+\sqrt{{\gamma_{\rm c}}})\|F\|_2\|u_{\rm c}\|_{H^1}.
\end{equation*}
\end{lemma}
\begin{proof}
We denote by $A=\sqrt{{\gamma_{\rm c}}}$ for simplicity. 
To compute the series \eqref{eqn:H_L-expand}, we note that the explicit integral kernel of $H_{\rm c}^{-1}=(-\Delta+A^2)^{-1}$ is
\[
    (H_{\rm c}^{-1} f)(x) =  \frac{\sinh(Ax)}{A\sinh(A)}\int_0^1 \sinh(A(1-y))f(y) dy - \frac{1}{A}\int_0^x \sinh(A(x-y))f(y) dy .
\]
We integrate by parts to get
\begin{multline*}
  \int_0^x \sinh(A(x-y))f(y){\wt V_{\om}}(y) dy  
    =  \sinh(A(x-y))f(y)F(y) \mid_{y=0}^{y=x} \\
     + \int_0^x \big(A\cosh(A(x-y))f(y)-\sinh(A(x-y))  \nabla f(y)\big)F(y) dy. \\
    =  \int_0^x \big(A\cosh(A(x-y))f(y)-\sinh(A(x-y)) \nabla f(y)\big)F(y) dy.  
\end{multline*}
For notation simplicity, let 
\[s_x(y)  := \sinh(A(x-y))H(x-y), \ \ \ 
    c_x(y)  := \cosh(A(x-y))H(x-y), \]
where $H$ is the Heaviside function. We can rewrite
\[
    (H_{\rm c}^{-1}{\wt V_{\om}}f)(x) = \frac{\sinh(Ax)}{\sinh(A)} \lan c_1 f -A^{-1} s_1  \nabla f, F \ran - \lan c_x f -A^{-1} s_x   \nabla f, F \ran.
\]

This allows us to complete the following estimate.
\begin{lemma} \label{lem:useful-series-bound}
Assume that $f \in H^1([0,1])$, then
\begin{equation*}
    \| H_{\rm c}^{-1}{\wt V_{\om}}f \|_{H^1} \lesssim (1+A)\|f\|_{H^1}\|F\|_2
\end{equation*}
\end{lemma}
\begin{proof}
We prove the bound for the derivative term in $H^1$ only since the $L^2$ term is similar. Taking a derivative in $x$, we see that
\begin{multline*}
   \nabla (H_{\rm c}^{-1}{\wt V_{\om}}f)(x) =  \frac{\cosh(Ax)}{\sinh(A)} \lan A c_1 f - s_1 \nabla f, F \ran - \lan A s_x f - c_x \nabla f, F \ran  
     - f(x) F(x) \\
    =  A\lan I_1, fF\ran - \lan I_2, \nabla f F\ran - f(x)F(x),  
\end{multline*}
where
\begin{align*}
    I_1 =& \frac{\cosh(Ax)}{\sinh(A)} \cosh(A(1-y)) - \sinh(A(x-y)) H(x-y) \\
    I_2 = & \frac{\cosh(Ax)}{\sinh(A)} \sinh(A(1-y)) - \cosh(A(x-y)) H(x-y) . 
\end{align*}
To proceed, we estimate the $L^\infty$ norm (in $y$ first, then in $x$) of $I_1$ and $I_2$. However, we will prove the case for $I_2$ only as that of $I_1$ is similar. Let $x, y \in [0,1]$. If $y < x$,
\begin{align*}
    2I_2 =& \frac{(e^{Ax}+e^{-Ax})(e^{A-Ay}-e^{-A+Ay})}{e^A-e^{-A}} - e^{Ax-Ay}-e^{-Ax+Ay} \\
    =& e^{A(x-y)} \left( \frac{(1+e^{-2Ax})(1-e^{-2A(1-y)})}{1-e^{-2A}} -1 - e^{-2A(x-y)} \right).
\end{align*}
If $A$ is small, clearly $I_2 \lesssim 1$. If $A$ is large, since $y < x$, we see that $I_2$ can be bounded by leading order terms in the Taylor expansion of its right hand side:
\begin{multline*}
    I_2 \lesssim  e^{A(x-y)}\big(e^{-2Ax}+e^{-2A(1-y)}+e^{-2A} - e^{-2A(x-y)} \big) \\
    =  e^{-A(x+y)} + e^{-A(2-x-y)}+e^{-A(2-x+y)} + e^{-A(x-y)}  
    \lesssim  1.
\end{multline*}

If $y > x$, then the heavisdie function is $0$. So
\[
  2I_2 = e^{A(x-y)} \left( \frac{(1+e^{-2Ax})(1-e^{-2A(1-y)})}{1-e^{-2A}} \right) 
    \lesssim e^{A(x-y)} 
    \lesssim1.
\]
Similar computation also implies $I_1\lesssim 1$. 
By Sobolev's inequality in 1D, it follows that
\[
    |\nabla (H_{\rm c}^{-1} {\wt V_{\om}} f)| \lesssim (1+A) \|f\|_{H^1}\|F\|_2.
\]
Hence,
\[
    \|H_{\rm c}^{-1} {\wt V_{\om}} f\|_{H^1} \lesssim (1+A)\|f\|_{H^1}\|F\|_2, 
\]
as claimed. The proof of Lemma \ref{lem:useful-series-bound} is complete.
\end{proof}

It follows by Lemma \ref{lem:useful-series-bound} and equation \eqref{eqn:H_L-expand} that
\[
    \|H_L^{-1} 1 - u_{\rm c}\|_{H^1} \lesssim \sum_{n\geq 1} (1+A)^n\|F\|_2^n\|u_{\rm c}\|_{H^1}.
\]
The proof of Lemma \ref{lem:u-T-first-order} is complete.
\end{proof}

Finally, we show that $F$ can be controlled.

\begin{lemma} \label{lem:F-bound-high-prob}
There is a constant $C_1$ only depending on the range of $\omega$ and a constant $C_2>0$ only depending on $\E(\omega)$. For any $0 < a < \frac{1}{2}$,
\[
    \|F\|_2 \le C_2 \gamma_{\rm c} L^{-a}
\]
with probability at least $1-e^{-C_1\, L^{(1-2a)^2}}$ as $L \rightarrow \infty$.
\end{lemma}
\begin{proof}
Let $x \in [0, 1]$ be an integer multiple of $1/L$: $x = n/L$ for $n \in \Z\cap [0,L]$. Let $V_{\om}$ be the piecewise constant potential with i.i.d. random coefficients $\omega_j$ as in \eqref{eq:V}. We assume $\omega_j$'s are nonegative and bounded from above. Hence, its expectation is finite and positive: $0<\E(\omega)<\infty$. Without loss of generality, we assume $\E(\omega)=1$. 
Recall the definitions of $V_{\om,L}$ and $\gamma_{\rm c}$ in \eqref{eqn:H-L-def} and \eqref{eq:gamma-c-def}, we have $\gamma_{\rm c}=kL^2$ and $V_{\om,L}(x)=kL^2(V_{\om}(Lx)-1)$.  
Let $S_n=\omega_1\cdots+\omega_n$. By definition \eqref{eqn:F-anti-d-def},
\begin{multline*}
    F\left(\frac{n}{L}\right)=\int_0^{n/L} {\wt V_{\om}}(y) dy= \gamma_{\rm c} \int_0^{n/L} { V_{\om}}(Ly)-1 \,  dy\\=\frac{ \gamma_{\rm c}}{L }\int_0^{n} { V_{\om}}(y)-1 \,  dy
    =\frac{ \gamma_{\rm c}}{L }\big(S_n-\E(S_n)\big).
\end{multline*}
Using a Riemann sum approximation, it follows that if $L$ is sufficiently large,
\begin{multline*}
     \int_0^1 |F(x)|^2 dx \leq  \frac{2}{L}\sum_{n=1}^L |F(n/L)|^2 \\\le \frac{2}{L}\sum_{n=1}^L \left(\frac{ \gamma_{\rm c}}{L }\big(S_n-\E(S_n)\big)\right)^2 
    =   \frac{ 2\gamma_{\rm c}^2}{L^3 } \sum_{n=1}^L \big(S_n-\E(S_n)\big)^2.
\end{multline*}
Since $\omega$ is bounded from above, $|S_n-\E(S_n)|\lesssim n \le L$. Fix $0<a<1/2$ and let $n_0=L^{1-2a}$. Then 
\begin{multline}
     \int_0^1 |F(x)|^2 dx \lesssim    \frac{ 2\gamma_{\rm c}^2}{L^3 } \sum_{n=1}^{n_0} L^2+\frac{ 2\gamma_{\rm c}^2}{L^3 } \sum_{n=n_0}^L \big(S_n-\E(S_n)\big)^2\\
   \le     \gamma_{\rm c}^2 L^{-2a}+\frac{ 2\gamma_{\rm c}^2}{L^3 } \sum_{n=n_0}^L \big(S_n-\E(S_n)\big)^2. \label{eq:tmp91}
\end{multline}
Let
\[
    E_{n} = \Big\{\,  \left|S_n-\E(S_n) \right| \ge  n^{1-a} \,  \Big\}.
\]
Since $\omega_1,\cdots,\omega_n$ are bounded independent random variables, Chernoff–Hoeffding's inequality (see e.g. \cite{H63}) implies that
\begin{equation}\label{eq:PEL}
    \P\left( E_{ n} \right) \le e^{-C\, \frac{n^{2(1-a)}}{n}}=e^{-C\, n^{1-2a} }
\end{equation}
for some constant $C$ only depends on the range of $\omega$. 
Let 
\begin{equation*}
    \mathcal E_L=\Big(E_{n_0}  \cup E_{n_0+1}  \cdots \cup E_{L} \Big)^C.
\end{equation*}
For $0<a<1/2$, we note that
\begin{equation}\label{eq:calEL}
    \P(  \mathcal E_L^C)=\P(E_{n_0}  \cup E_{n_0+1}  \cdots \cup E_{L} ) 
  \le \sum_{n=n_0}^\infty e^{-C\, n^{1-2a}}\lesssim e^{-C\, n_0^{1-2a}}=e^{-C\, L^{(1-2a)^2}}
\end{equation}
approaches $0$ as $L \rightarrow \infty$. 
On the set $\mathcal E_L$, the last sum in \eqref{eq:tmp91} can be bounded by
\[ \frac{ 2\gamma_{\rm c}^2}{L^3 } \sum_{n=n_0}^L \big(S_n-\E(S_n)\big)^2\le \frac{ 2\gamma_{\rm c}^2}{L^3 }\, L \, (L^{1-a})^2 \le 2\gamma_{\rm c}^2 L^{-2a}. \]
Putting all together, 
\begin{equation*}
    \|F\|_{2} \lesssim \gamma_{\rm c} L^{-a} 
\end{equation*}
on the set $\mathcal E_L$ with $\P(\mathcal E_L)\ge 1-e^{-C\, L^{(1-2a)^2}}$ and $L$ is sufficiently large. Thus, we have the proved Lemma \ref{lem:F-bound-high-prob}. 
\end{proof}

We now complete the proof of Lemma \ref{lem:u_L-main-lemma}. 
For $1+\gamma_{\rm c}<CL^{\beta}$, let $a=\beta+1/4$. Combing Lemma \ref{lem:u-c-lemma}, Lemma \ref{lem:u-T-first-order}, Lemma \ref{lem:F-bound-high-prob} with this choice of $a$, on the set $\mathcal E_L$, we get
\begin{equation*}
   \|u_L - u_{\rm c}\|_{H^1} \lesssim (1+\sqrt{{\gamma_{\rm c}}})\gamma_{\rm c}L^{-a} \frac{1+\sqrt \gamma_{\rm c}}{\gamma_{\rm c}}\lesssim (1+\gamma_{\rm c})L^{-a}. 
\end{equation*}
The explicit formula \eqref{eq:uc-sup} of $\max u_{\rm c}$ implies $\max u_{\rm c}\ge 8+\gamma_{\rm c}$. Combined with the Sobolev estimate $\|f\|_\infty \leq C\|f\|_{H^1}$ in 1D, one gets on $\mathcal E_L$ 
\begin{equation}\label{eq:tmp96}
 \Big | \frac{\max_{x\in[0,1]} u_L}{\max_{x\in[0,1]} u_{\rm c}}-1 \Big| \lesssim \frac{1}{\max u_{\rm c}}\, \|u_L - u_{\rm c}\|_{H^1}\lesssim  (1+\gamma_{\rm c})^2L^{-a} \lesssim L^{2\beta-a}=L^{\beta-1/4}   
\end{equation}
which completes the proof for \eqref{eq:uc-uL-probH}. 

Finally, let 
\[\mathcal E_\infty=\liminf_{L\to \infty} \mathcal E_L:=\bigcup_{n=1}^\infty\bigcap_{L=n}^\infty \mathcal E_L.\]
On $\mathcal E_\infty$, 
 there is $n_*$ such that  \eqref{eq:tmp96} holds for all $L\ge n_\ast$. Taking the limit as $L\to \infty$ gives
\eqref{eq:uc-uL-prob1}. 

By \eqref{eq:PEL} and \eqref{eq:calEL}, \[\sum_{L=1}^\infty\P (\mathcal E_L^C)<\infty.\]
As a direct consequence of the Borel–Cantelli lemma, the set 
\[ \mathcal E_\infty^C=\bigcap_{n=1}^\infty\bigcup_{L=n}^\infty \mathcal E_L^C\]
has probability zero, i.e., $\mathcal E_\infty$ has probability one.
\end{proof}

\subsubsection{Estimates for the ground state energy} \label{subsec:weak-p-lambd-estimate}
Recall that $\gamma_{\rm c}$ is given in \eqref{eq:gamma-c-def} and $L$ denotes the length of the underlying domain $[0,L]$
\begin{lemma} \label{lem:HL-lambda}
Assume that 
\[
   1+ \gamma_{\rm c}<C L^\beta
\]
for some $C>0$ and $\beta < \frac{1}{2}$. Then there are constants $C_1,C_2>0$ such that the ground state eigenvalue, $\lambda$, of $H_L$ (see \eqref{eqn:H-L-def}) satisfies
\begin{equation}\label{eq:lambda-probH}
    \left|\frac{\lambda}{\gamma_{\rm c} +\pi^2} -1\right| \le C_2 L^{-(1/2-\beta)/2}
\end{equation}
with probability  $1-e^{-C_1\, L^{(1/2-\beta)^2}}$ as $L \rightarrow \infty$.

As a direct consequence, 
\begin{equation}\label{eq:lambda-prob1}
   \lim_{L\to\infty}  \frac{\lambda}{\gamma_{\rm c} +\pi^2} =1 
\end{equation}
with probability one. 
\end{lemma}
\begin{proof}
We decompose $H_L = H_{\rm c} + {\wt V_{\om}}$ via \eqref{eqn:H-T-decomp} as before. 
Let $\psi \in H^1([0,1])$ satisfy the  Dirichlet boundary conditions. It follows that
\[
    \lan \psi, H_L \psi \ran =  \lan \psi, H_{\rm c} \psi \ran +\lan \psi, {\wt V_{\om}} \psi \ran .
\]
Integrating by part and using $F$ in \eqref{eqn:F-anti-d-def}, we see that
\[
    \lan \psi, {\wt V_{\om}} \psi \ran =  -2\Re \int_0^1 F  \bar \psi \nabla \psi.
\]
Using Cauchy-Schwartz and Sobolev embedding, we see that
\begin{multline*}
     |\lan \psi, {\wt V_{\om}}  \psi \ran| \leq  \|F\|_2\|\nabla \psi\|_2\|\psi\|_\infty\lesssim \|F\|_2\big(\|\nabla \psi\|_2^2+\|\psi\|_\infty^2\big) \\
     \lesssim \|F\|_2\big(\|\nabla \psi\|_2^2+\|\psi\|_{H^1}^2\big) \lesssim \|F\|_2\big(\|\nabla \psi\|_2^2+\|\psi\|_{2}^2\big)
   =  \|F\|_2 \lan \psi, (1-\Delta) \psi\ran.
\end{multline*}
It follows that
\[|\lan \psi, {\wt V_{\om}}  \psi \ran|
    \lesssim  \|F\|_2 ( \gamma_{\rm c}^{-1}+ 1) \lan \psi, H_{\rm c} \psi\ran. \]
    Hence,
\begin{equation}\label{eq:tmp99}
    \lan \psi, H_L \psi \ran =  \lan \psi, H_{\rm c} \psi \ran \big(1 + O(\gamma_{\rm c}^{-1}+ 1)\|F\|_2\big).
\end{equation}
The proof of \eqref{eq:lambda-probH} of Lemma \ref{lem:HL-lambda} is completed as a result of \eqref{eq:tmp99}, Lemma \ref{lem:F-bound-high-prob} with the choice of $a=\beta+1/2$ and the explicit expression of the first eigenvalue of $H_{\rm c}$. 

The proof for \eqref{eq:lambda-prob1} is  again based on the probability estimate of \eqref{eq:lambda-probH} and the Borel–Cantelli lemma, which is similar to the proof of \eqref{eq:uc-uL-prob1}. We omit the details here.   
\end{proof}

The proof of Theorem \ref{thm:33} is completed as a result of \eqref{eqn:lambda-u-lambda_L-u-L-equiv}, Lemma \ref{lem:u_L-main-lemma} and \ref{lem:HL-lambda}.

\subsection{Heuristic arguments for excited states energies for the Bernoulli case}\label{sec:excited}

In this section, we will discuss the observation \eqref{eq:first} for the  excited states $n\ge 2$. We restrict ourselves to $H=-\Delta+kV_{\om}$ given as \eqref{eq:H} with a Bernoulli-piecewise potential $V_{\om}$ taking values $0,1$, i.e., 
\begin{equation*}
    p=\P(\om=0),\ \ 1-p=\P(\om=1). 
\end{equation*}
Let $\lambda_n$ be the $n$-th smallest eigenvalue of $H$ under Dirichlet boundary conditions on $[0,L]$.
Denote the effective potential by
\[W=\frac{1}{u},\]
and the $n$-th local minimum of $W$ by $W_n$. 
With a lot of numerical evidence in Section \ref{sec:numerics}, we conclude 
\begin{equation}\label{eq:appro}
    \lambda_n\approx \frac{\pi^2}{8}W_n.
\end{equation}
 In Theorem \ref{thm:1}, we provide the rigorous proof of the approximation \eqref{eq:appro} for the ground state case when $n=1$. We now further justify  heuristically the approximation for the excited states when $n\ge 2$.

The sets $\{x|kV_{\om}(x)=k\}$ and $\{x|kV_{\om}(x)=0\}$ consist of finitely many intervals (connected components). We may call these intervals $k$-walls and zero wells, respectively. 
We denote by $I_i$ the $i$-th zero well with length $L_i$, arranged non-increasingly with respect to the length: $L_1\ge L_2\ge L_3\cdots$. As $k\to \infty$, $-\Delta+kV_{\om}$ can be approximated by the direct sum of (negative) free Laplacian $-\Delta$ on $I_i$ with  Dirichlet boundary conditions on $\partial I_i$. The energy levels of  $-\Delta$ on $I_i$ with  Dirichlet boundary conditions are simple:
\begin{equation}\label{eq:Eis}
    E_{i,s}=\frac{s^2\pi^2}{L_i^2},\ i=1,2,\cdots,\ s=1,2,\cdots.
\end{equation}
\begin{figure}[H]
    \centerline{
    \includegraphics[width=1.0\textwidth]{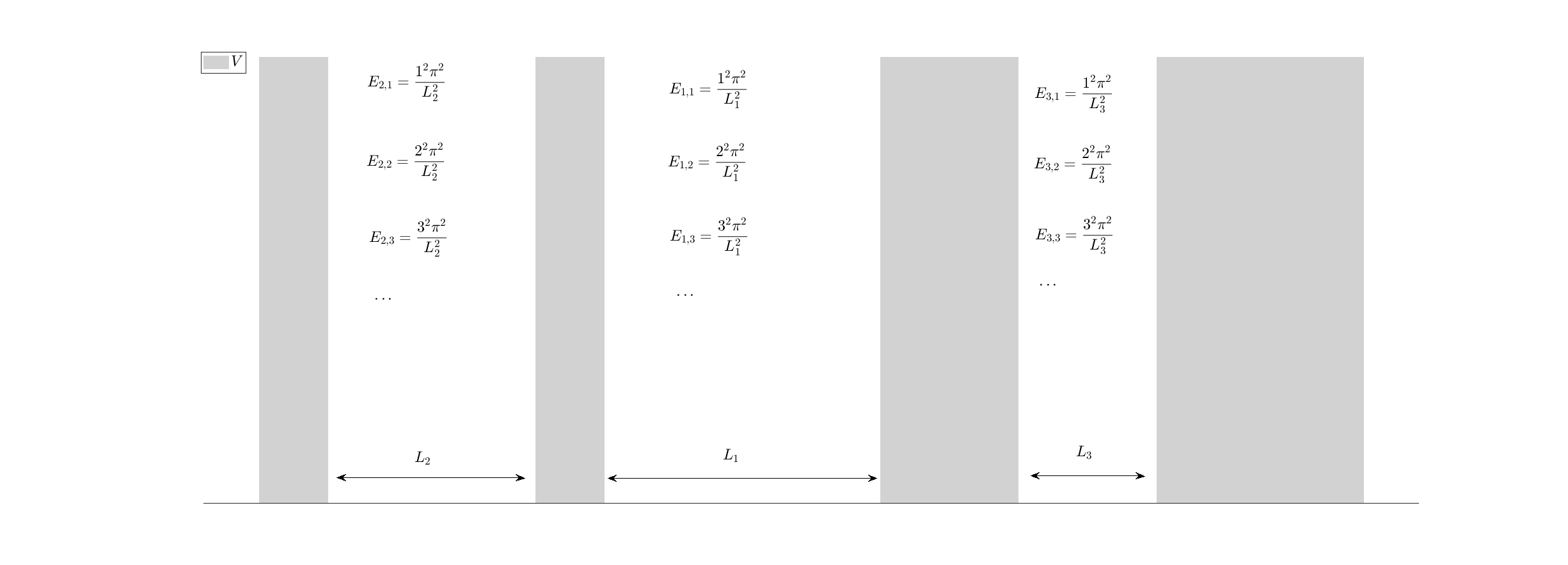} }
    \caption{The energy levels of $-\Delta$ for the first three longest zero wells.}
\end{figure}

Hence, the energy levels of $-\Delta+kV_{\om}$ can be approximated by the rearranging of $E_{i,s}$ in a  non-decreasing order. In particular, some bottom energy levels can be approximated by the first harmonics $E_{i,1}$ of \eqref{eq:Eis}:
\begin{equation*}
\lambda_i\approx \frac{\pi^2}{L_i^2},\ \ i=1,2,\cdots,i_0.
\end{equation*}

Let $u^i$ be the local landscape function for the free problem $-\Delta u^i=1$ on the $i$-th zero well $I_i$ with  Dirichlet boundary conditions. For a similar reason for the approximation of the eigenvalues, the restriction of the global landscape function $u$ on $I_i$ can be approximated by $u^i$, which implies 
\begin{equation*}
   W_i\approx \frac{1}{\max_{I_i} u}\approx \frac{1}{\max_{I_i} u^i}=\frac{8}{L_i^2},\ \ i=1,\cdots,i_0'.
\end{equation*}
Therefore, for excited states near the bottom of the spectrum, we have the approximation
\begin{equation}\label{eq:approx1}
    \frac{\lambda_i}{W_i}\approx \frac{\pi^2/L_i^2}{8/L_i^2}=\frac{\pi^2}{8}, \ i=1,\cdots,\min\{i_0,i_0'\}. 
\end{equation}

 In Section \ref{sec:numerics}, we will show numerical experiments to verify \eqref{eq:approx1}, and a generalized method to deal with eigenvalues contributed by the second, third, etc harmonics.


\section{Numerical experiments} \label{sec:numerics}

	In this section, we will display extensive numerical experiments to support our theory. Comparing with the notation $H=-\Delta+kV_{\om}$ we used in Section \ref{sec:theory}, we absorb the disorder strength $k$ into $V_{\om}$ in this section. More precisely, we will consider the 1-d Schr\"odinger operator $H=-\Delta+V_{\om}$ with a Bernoulli piecewise constant potential $V_{\om}$ on the domain $[0, L]$, with Dirichlet boundary conditions. Here $L$ is chosen as a  positive integer,
	and $[0, L]$ contains $L$ unit cells. The Bernoulli potential $V_{\om}$ is a piecewise constant potential as in \eqref{eq:V}. The $L$ random values of $V_{\om}$, chosen as either 0 or $V_{\max}$ with probability $p$ and $1-p$, are assigned to the $L$ unit cells independently. 
	Throughout this section, we still use $W=\dfrac{1}{u}$ to represent the effective potential. Denote  the global minimum of $W$ by $W_{\min}=\min \frac{1}{u}$.

	First we consider the domain [0, 10000], and the value of the potential $V_{\om}$ is either 0 or 10, each with probability 50\%. We test 100 different random realizations (Figure \ref{lambda1}): 
	
	\begin{figure}[H]
	 	\centerline
	 	{	\includegraphics[width=0.5\textwidth]{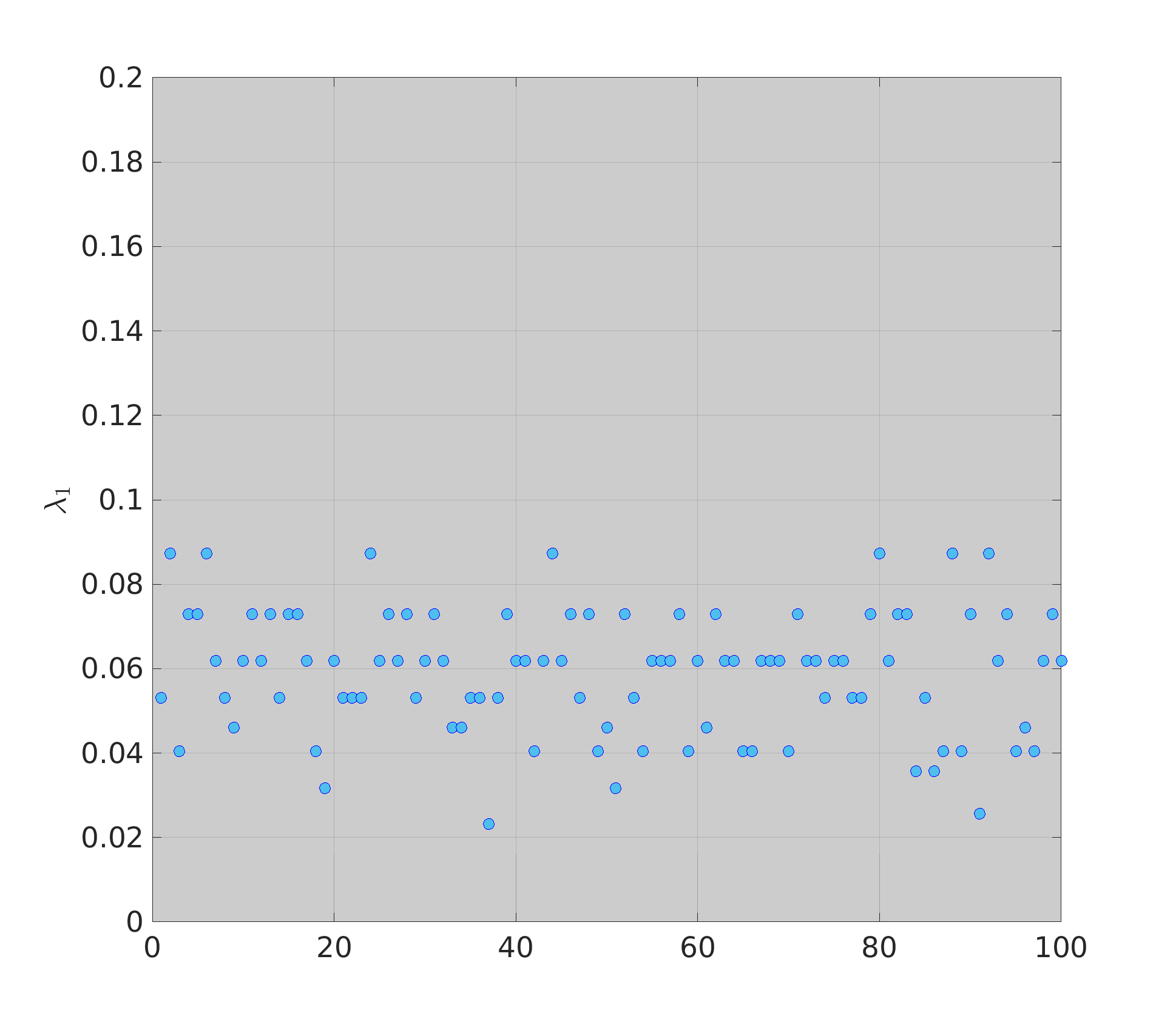}	}
	 	\caption{The ground state eigenvalues $\lambda_1$ from 100 independent realizations}
	 	\label{lambda1}
	 \end{figure}
	Then Figure \ref{BSeed} shows the ratio of  $\lambda_1$ over $W_{\min}$ and in each realization.  
		\begin{figure}[H]
		\centering
		\subfigure{
			\includegraphics[width=7.2 cm]{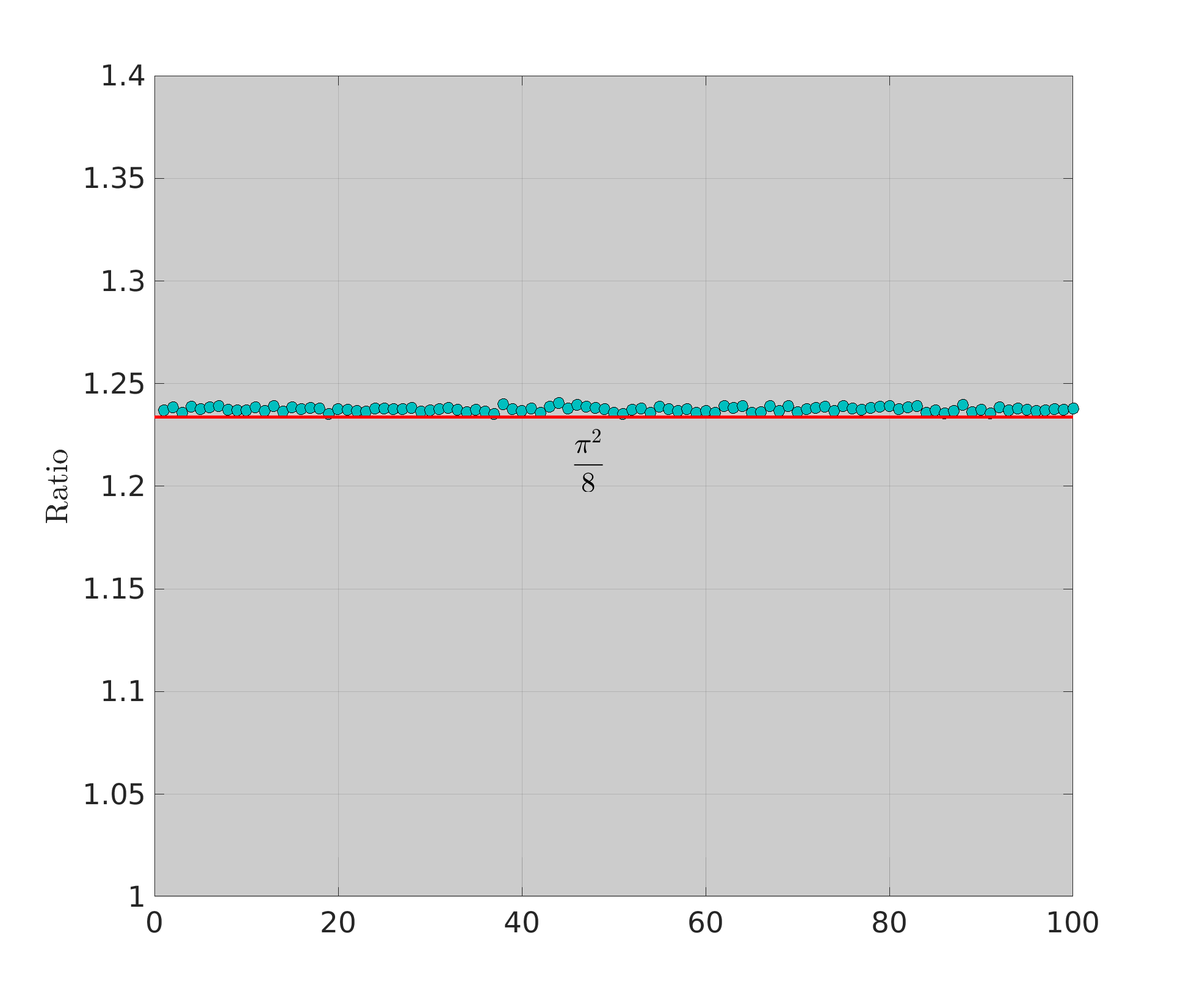}
			
		}
		\quad
		\subfigure{
			\includegraphics[width=7.2 cm]{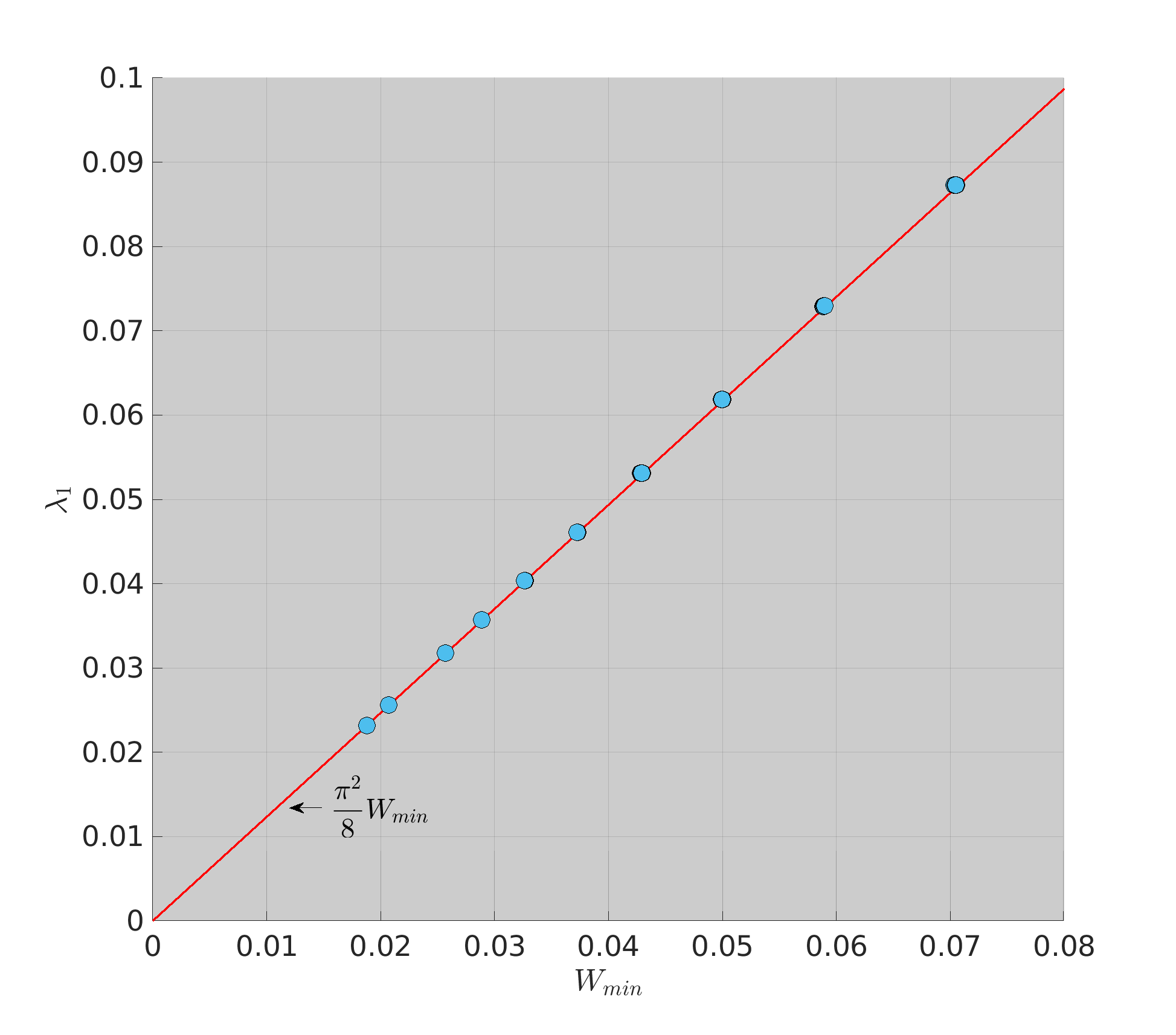}
		}
		\caption{The left plot shows the ratio from 100 random realizations. The right plot shows the first eigenvalues versus the corresponding $W_{\min}$ of 100 realizations. }	
		\label{BSeed}
	\end{figure}
	
	As is observed in Figure \ref{BSeed}, although the domain and the parameters of $V_{\om}$ are fixed,  $\lambda_1$ and $W_{\min}$ still depend on the specific realization. However, the ratio $\dfrac{\lambda_1}{W_{\text{min}}}$  always keeps close to $\dfrac{\pi^2}{8}$.

	Next, we test the ratio when $V_{\max}$ varies. Likewise, the Bernoulli potential $V_{\om}$ still involves 50\% 0 and 50\% $V_{\max}$, where  $V_{\max}$ varies from $2^{-36}\approx 1.455\times 10^{-11}$ to $2^{11}=2048$. We choose one realization with various $V_{\max}$ in the following case, where the domain is fixed as [0,1000].

	\begin{figure}[H]
		\centering
		\subfigure{
			\includegraphics[width=7.4 cm]{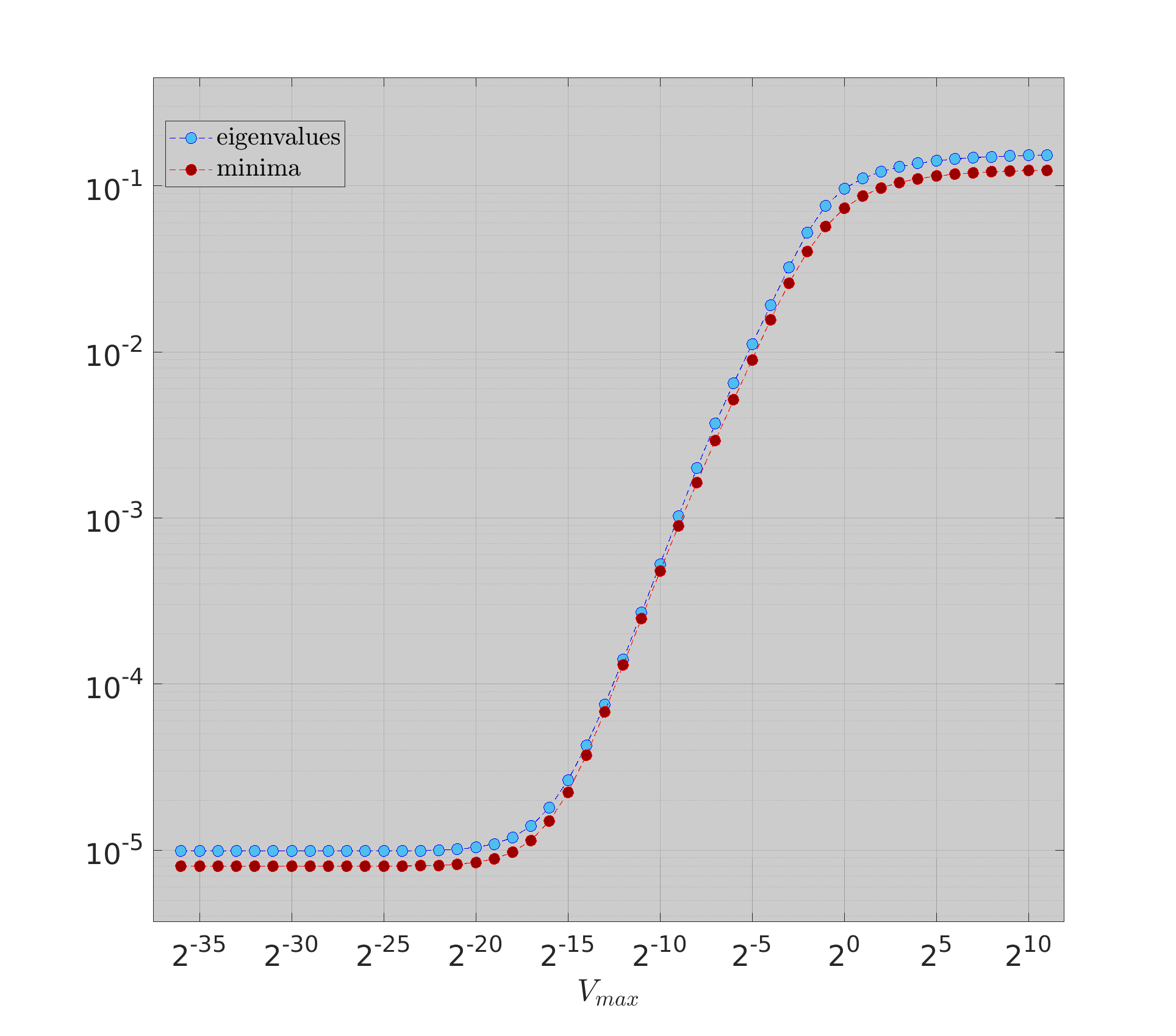}
			
		}
		\quad
		\subfigure{
			\includegraphics[width=7. cm]{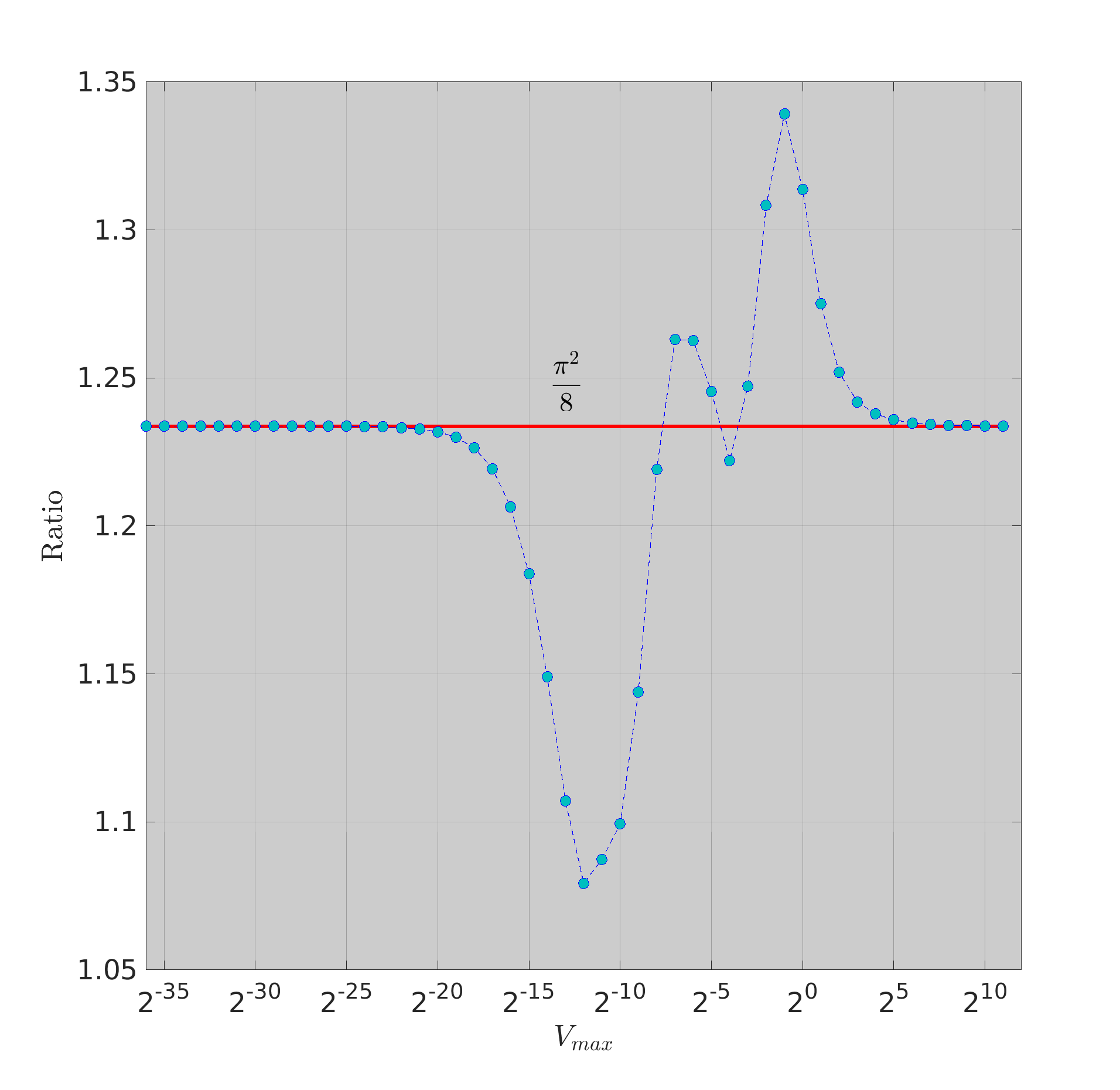}
		}
		\caption{The left plot displays a comparison of the first eigenvalue with the corresponding  $W_{\min}$ for different $V_{\max}$.  The right plot displays the ratio's dependence on $V_{\max}$, which tends to be $\dfrac{\pi^2}{8}$ when $V_{\max}$ is large or tiny.  }	
		\label{Vmax}
	\end{figure}
    Evidently, the behavior of $\dfrac{\lambda_1}{W_{\min}}$, shown in Figure \ref{Vmax}, is highly in accord with our theoretical statement: when $V_{\max}$ is close to 0 sufficiently, $\dfrac{\lambda_1}{W_{\min}}$ approaches to $\dfrac{\pi^2}{8}$, like the free Laplacian case. As $V_{\max}$ goes to infinity, $\dfrac{\lambda_1}{W_{\min}}$ gets back to $\dfrac{\pi^2}{8}$. Actually, $V_{\max}$ does not need to be sufficiently large in practical. From the right plot of Figure \ref{Vmax}, $\dfrac{\lambda_1}{W_{\min}}$ gets highly close to $\dfrac{\pi^2}{8}$ even when $V_{\max}$ is mildly large.

	To verify the ratio's dependence on the domain size $L$, in the following experiments, $V_{\max}$ is fixed and $L$ varies from $2^7=128$ to $2^{23}=8388608$. In Figure \ref{sizeL}, we consider two cases, in which two potentials with different probability and $V_{\max}$ are used. The first potential $V_{\om}$ is generated by choosing either 0 or 4 randomly with probability 70\% and 30\%
	, while in the second one, 0 and 100 are assigned randomly, with probabilities 50\%.
 	
 	Overall, both cases in Figure \ref{sizeL} support the theoretical result. As $L$ increases, both $\lambda_1$ and $W_{\min}$ get smaller, but the ratio $\dfrac{\lambda_1}{W_{\min}}$ converges to $\dfrac{\pi^2}{8}$. Although the increasing $L$ pushes the ratio to $\dfrac{\pi^2}{8}$ for various $V_{\max}$, larger $V_{\max}$ gives a faster convergence rate.
 	
	 \begin{figure}[H]
	 	\centering
	 	\subfigure{
	 		\includegraphics[width=7.2 cm]{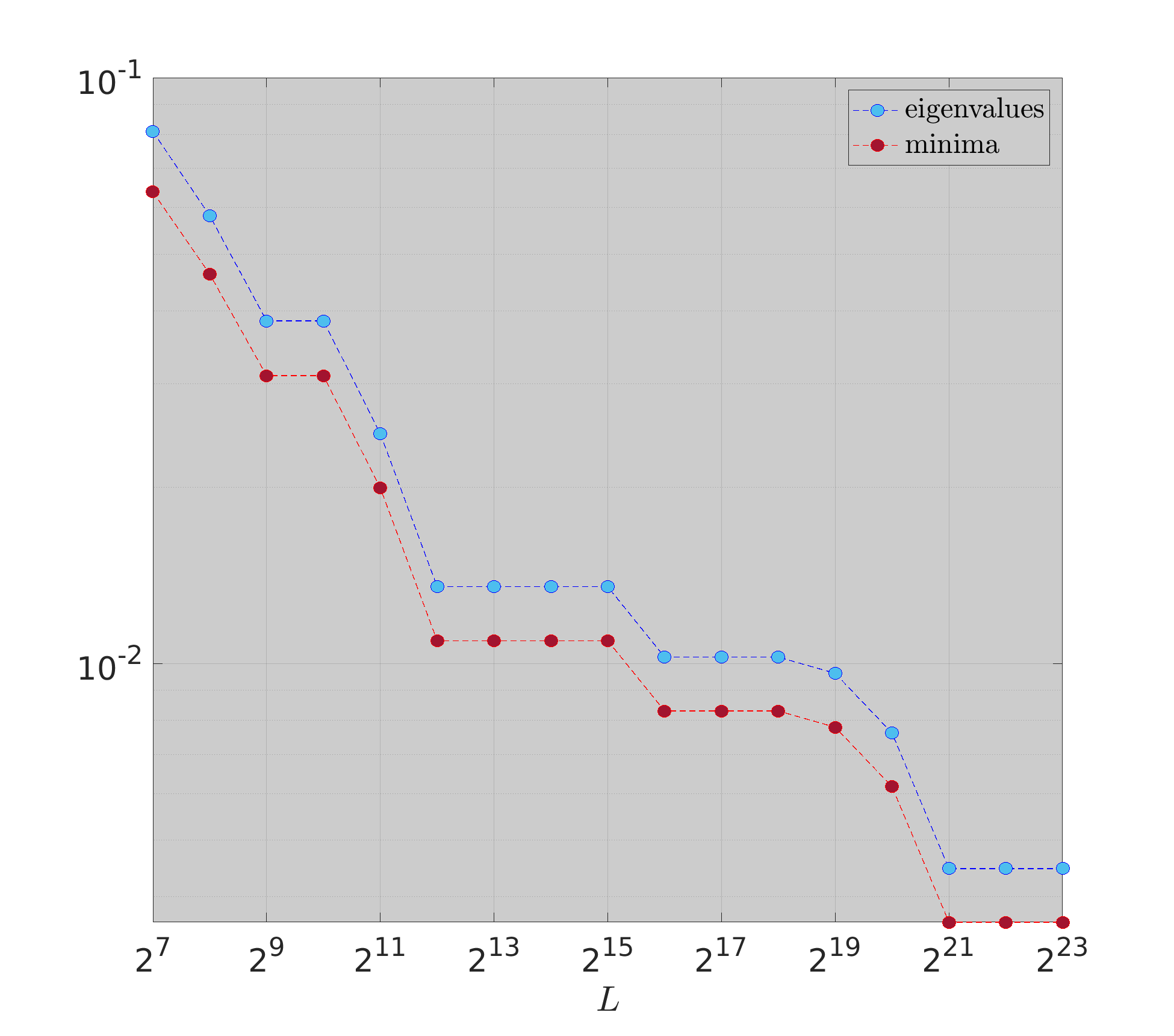}	
	 	}
	 	\quad
	 	\subfigure{
	 		\includegraphics[width=7.2 cm]{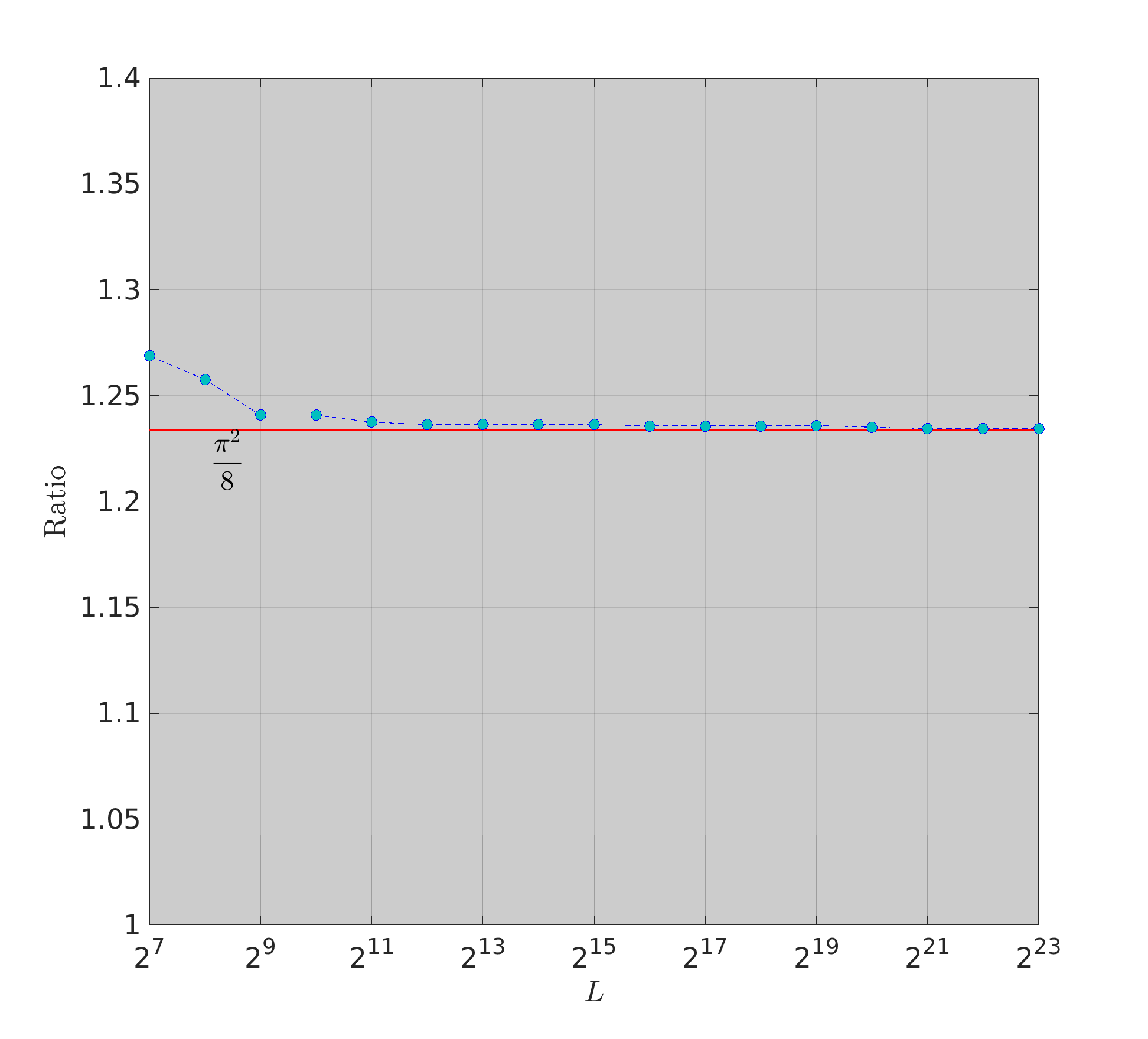}
	 	}
	 	\quad
	 	\subfigure{
	 		\includegraphics[width=7.2 cm]{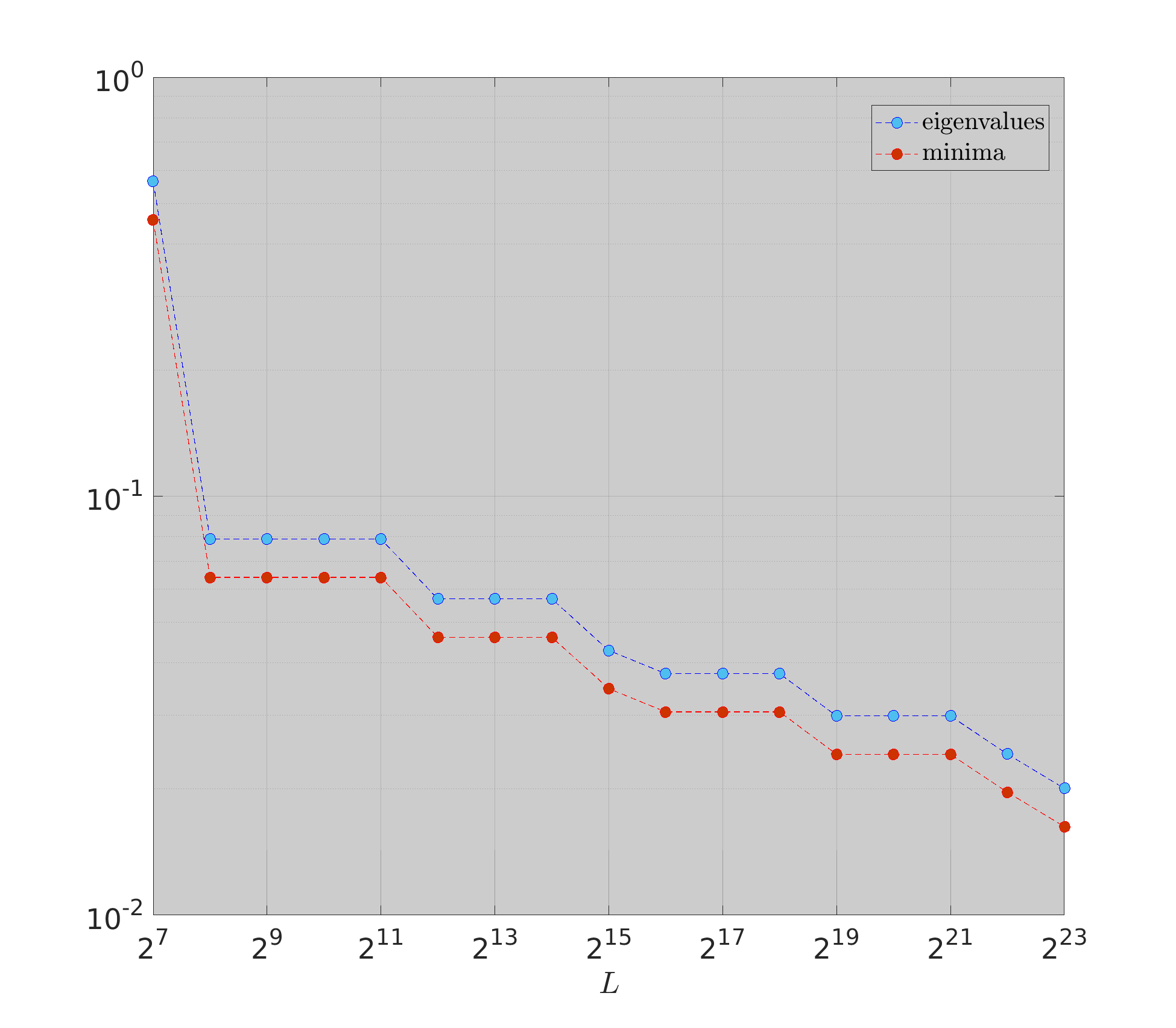}
	 	}	
	 	\quad
	 	\subfigure{
	 		\includegraphics[width=7.2cm]{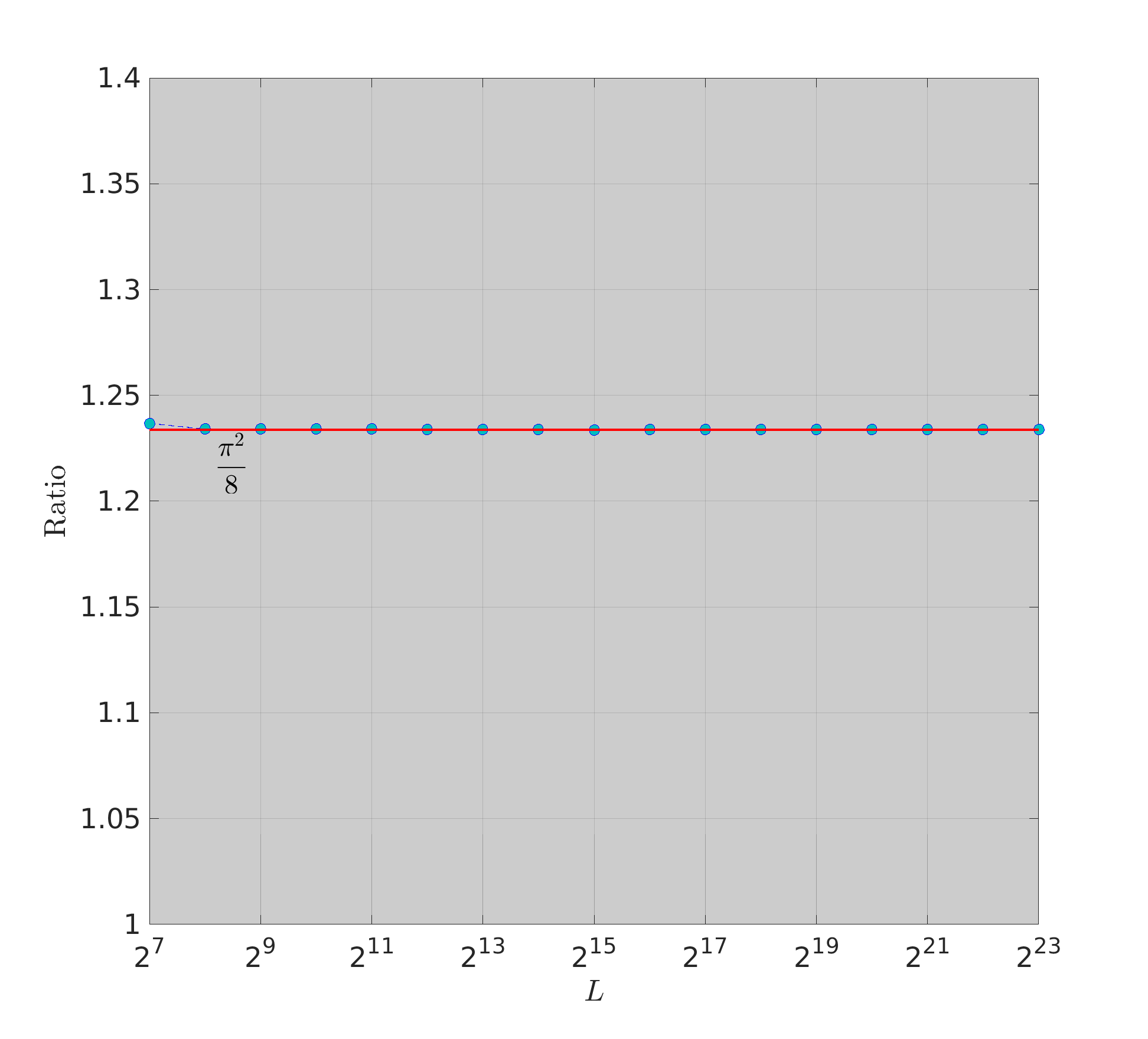}
	 	}
	 	\caption{The top row: Bernoulli potential with 70\% 0 and 30\% 4. The bottom row: Bernoulli potential with 50\% 0 and  50\% 100. The  left column displays a comparison of the first eigenvalue with the corresponding  $W_{\min}$ for different $L$.  The right column displays the corresponding ratio's dependence on $L$.  }	
	 	\label{sizeL}
	 \end{figure}

	Although, we only provide the rigorous proof for the first eigenvalue, in practical, the ratio actually can be extended to a large range of excited state eigenvalues and their associated local minima as in \eqref{eq:appro}. 
	With \eqref{eq:appro}, we can only compute the $n$-th local minimum and $\dfrac{\pi^2}{8}W_n$ to approximate $\lambda_n$, which is pretty cheap compared with solving eigenvalues directly. Figure \ref{100eigs} shows two different Bernoulli cases, in which we solve the first 100 eigenvalues and associated local minima. The corresponding ratio is very close to   $\dfrac{\pi^2}{8}$.
	
	\begin{figure}[H]
		\centering
		\subfigure{
			\includegraphics[width=7.2 cm]{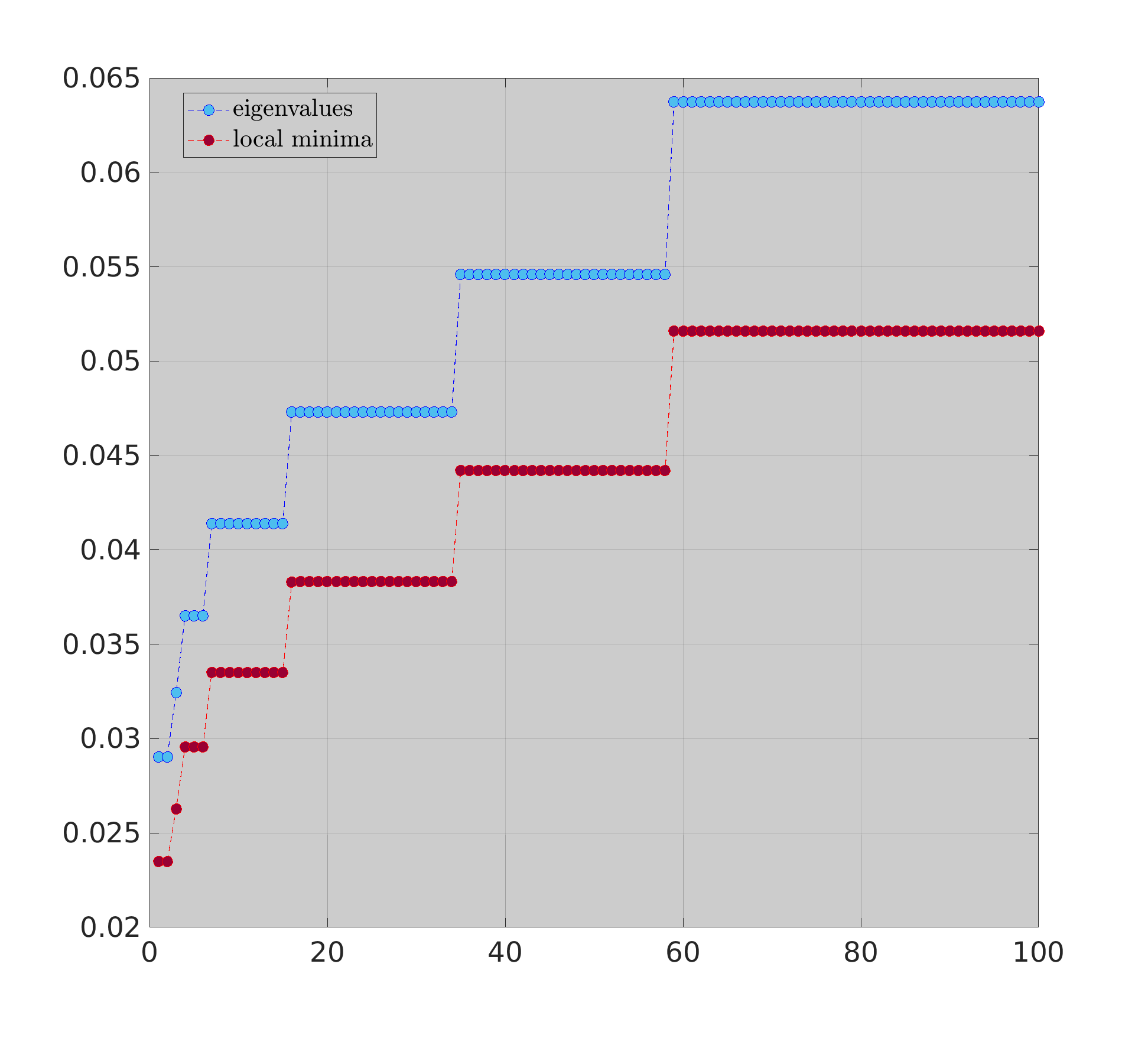}
			
		}
		\quad
		\subfigure{
			\includegraphics[width=7.2 cm]{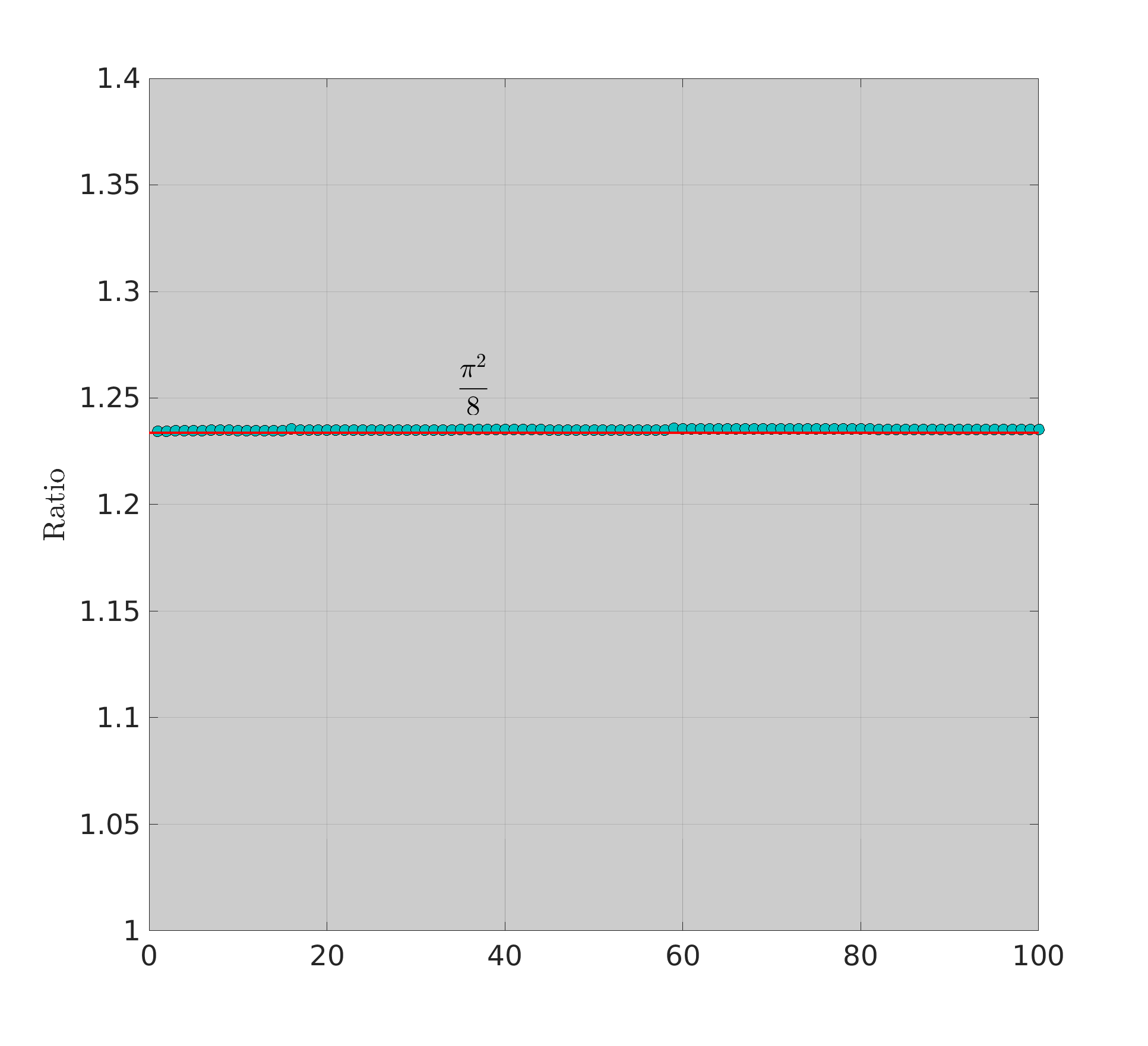}
		}
		\quad
     	\subfigure{
		\includegraphics[width=7.2 cm]{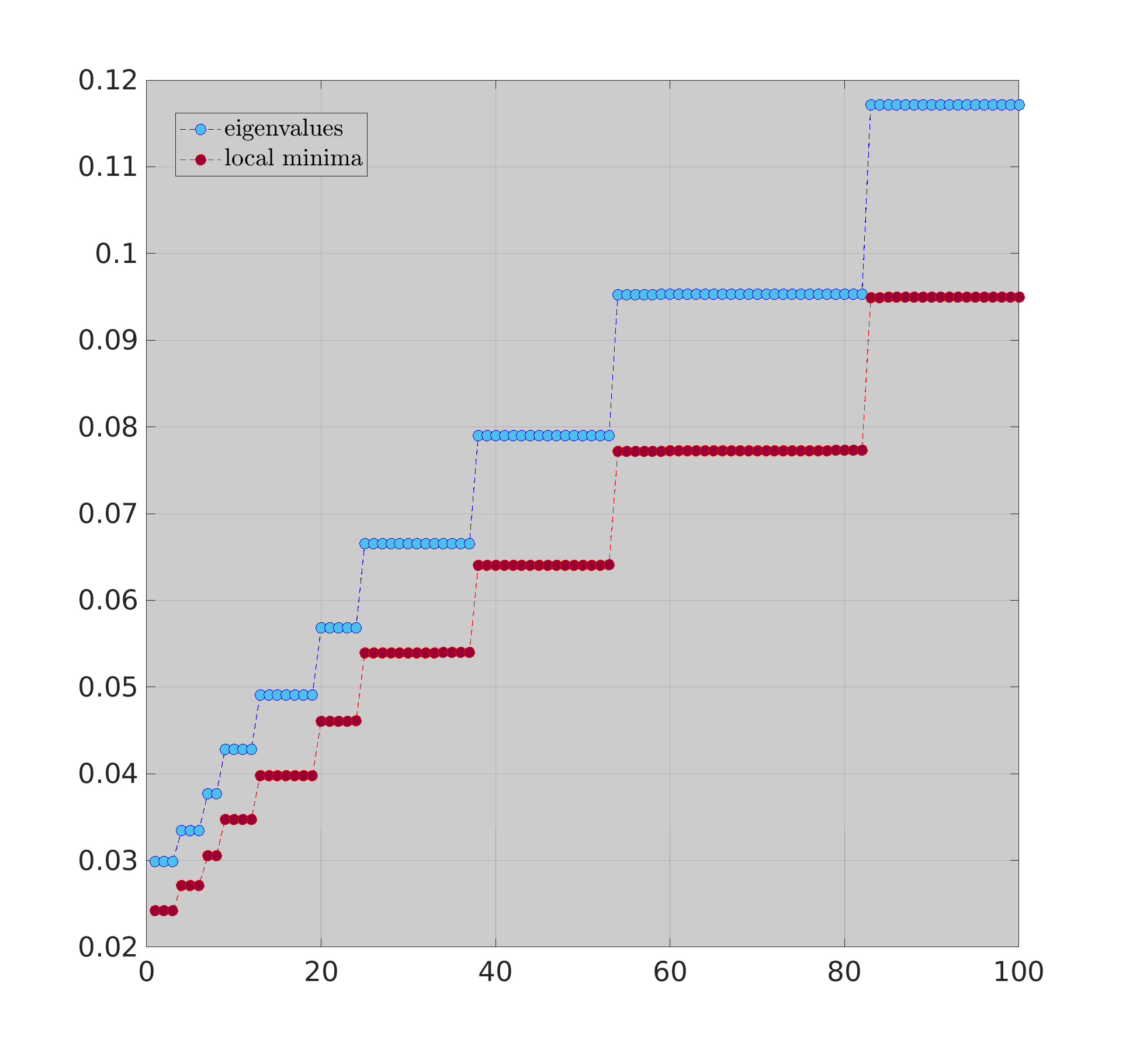}
	      }	
         \quad
        \subfigure{
        \includegraphics[width=7.2cm]{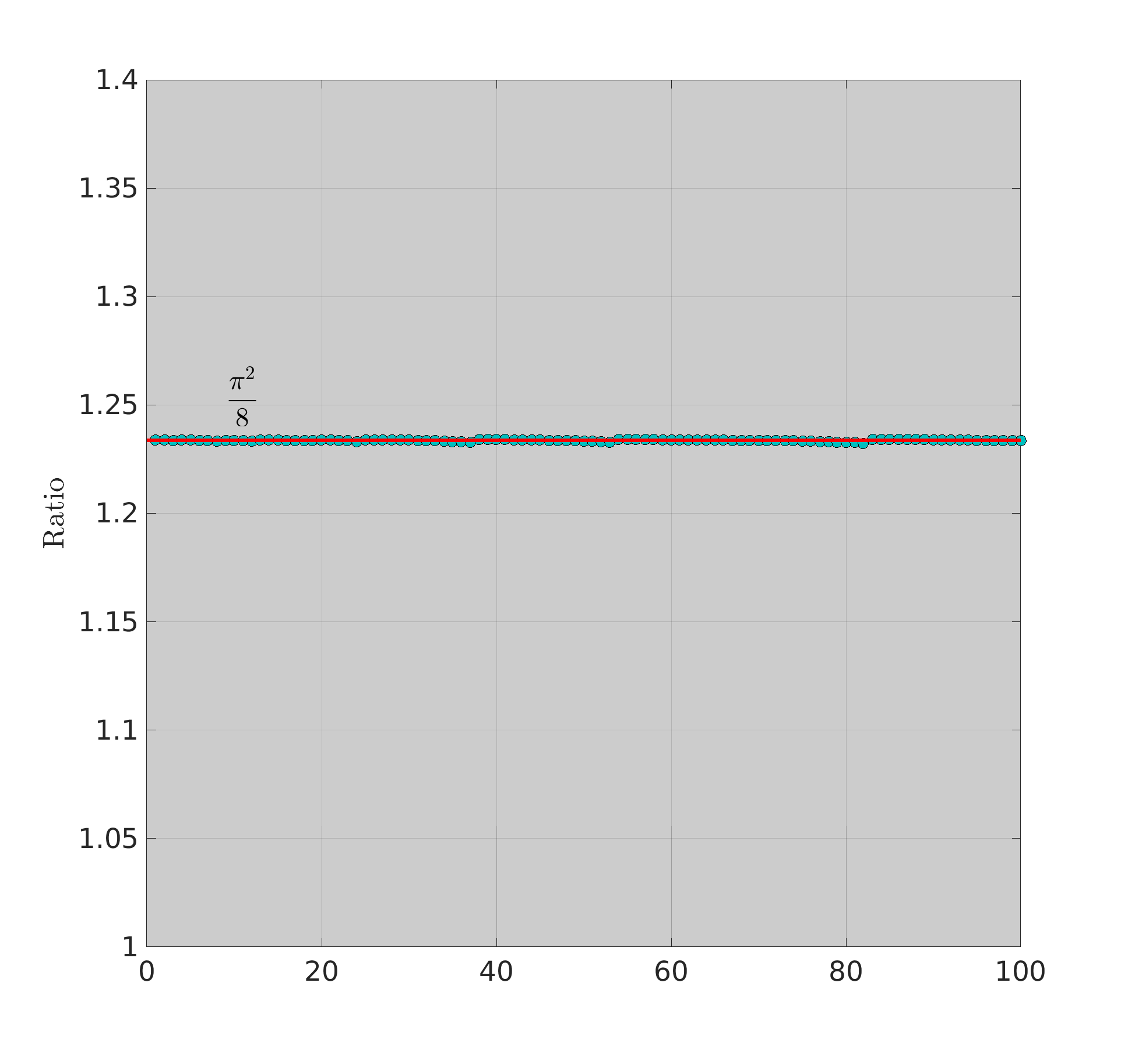}
              }
		\caption{The left column displays a comparison of the first 100 eigenvalues with the corresponding local minima of $W$.  The right column displays the corresponding ratio of  $\lambda_n$ and $W_n$ ($n=1,2,\cdots 100$) shown on the left.  The top row: Bernoulli potential with 50\% 0 and  50\% 20 on [0,1000000]. The bottom row: Bernoulli potential with 70\% 0 and  30\% 100 on [0,10000].   }	
		\label{100eigs}
	\end{figure}

 In some cases, the ratio $\dfrac{\lambda_n}{W_n}$ is away from  $\dfrac{\pi^2}{8}$ for some higher energy $\lambda_n$ and the associated local minimum $W_n$. For example, Figure \ref{Emore} shows one Bernoulli case: 30\% 20 and 70\% 0 on [0,10000]. The first 400 eigenvalues and corresponding local minima are solved:

	\begin{figure}[H]
	\centering
	\subfigure{
		\includegraphics[width=7.2 cm]{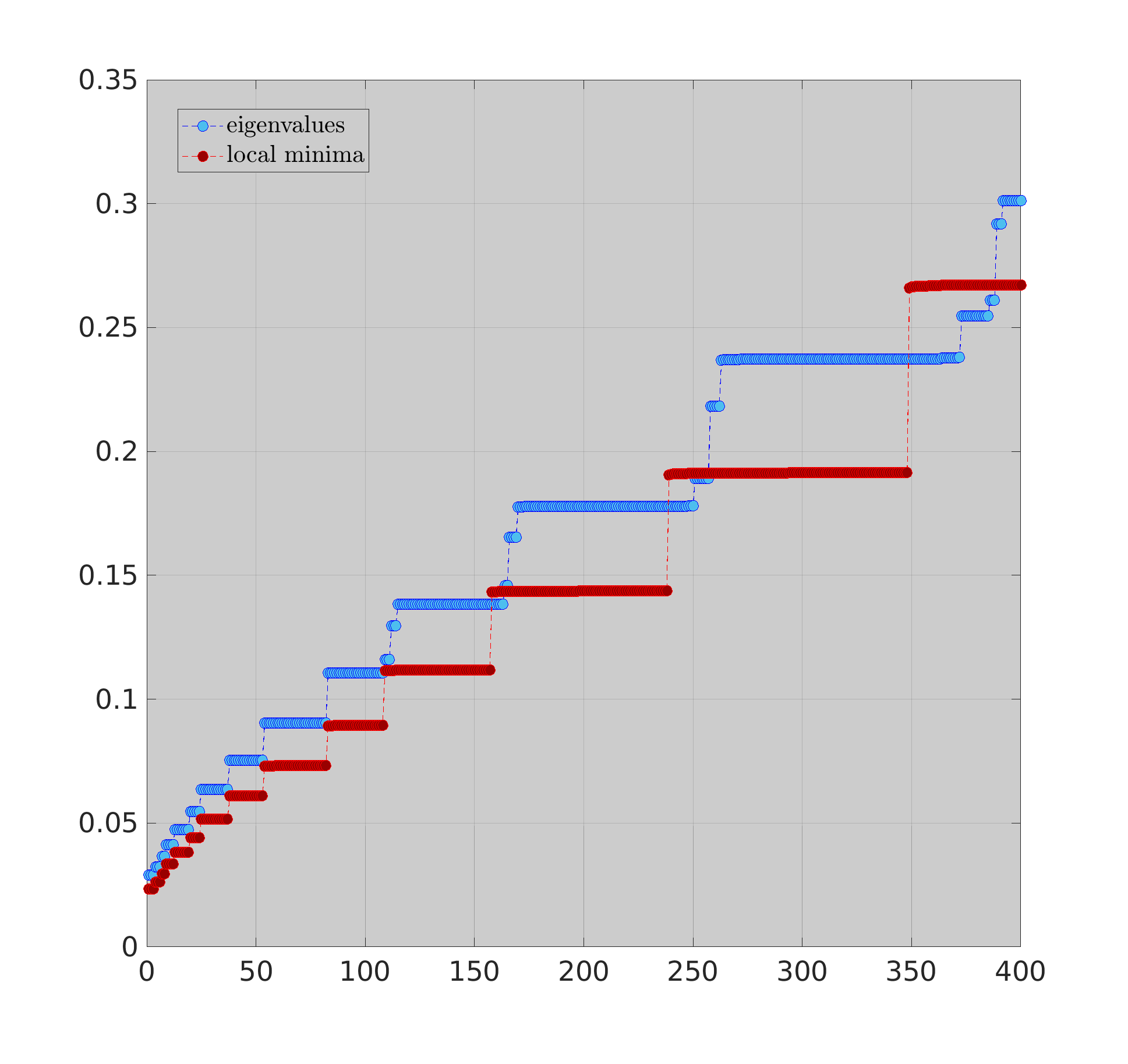}
		
	}
	\quad
	\subfigure{
		\includegraphics[width=7.2 cm]{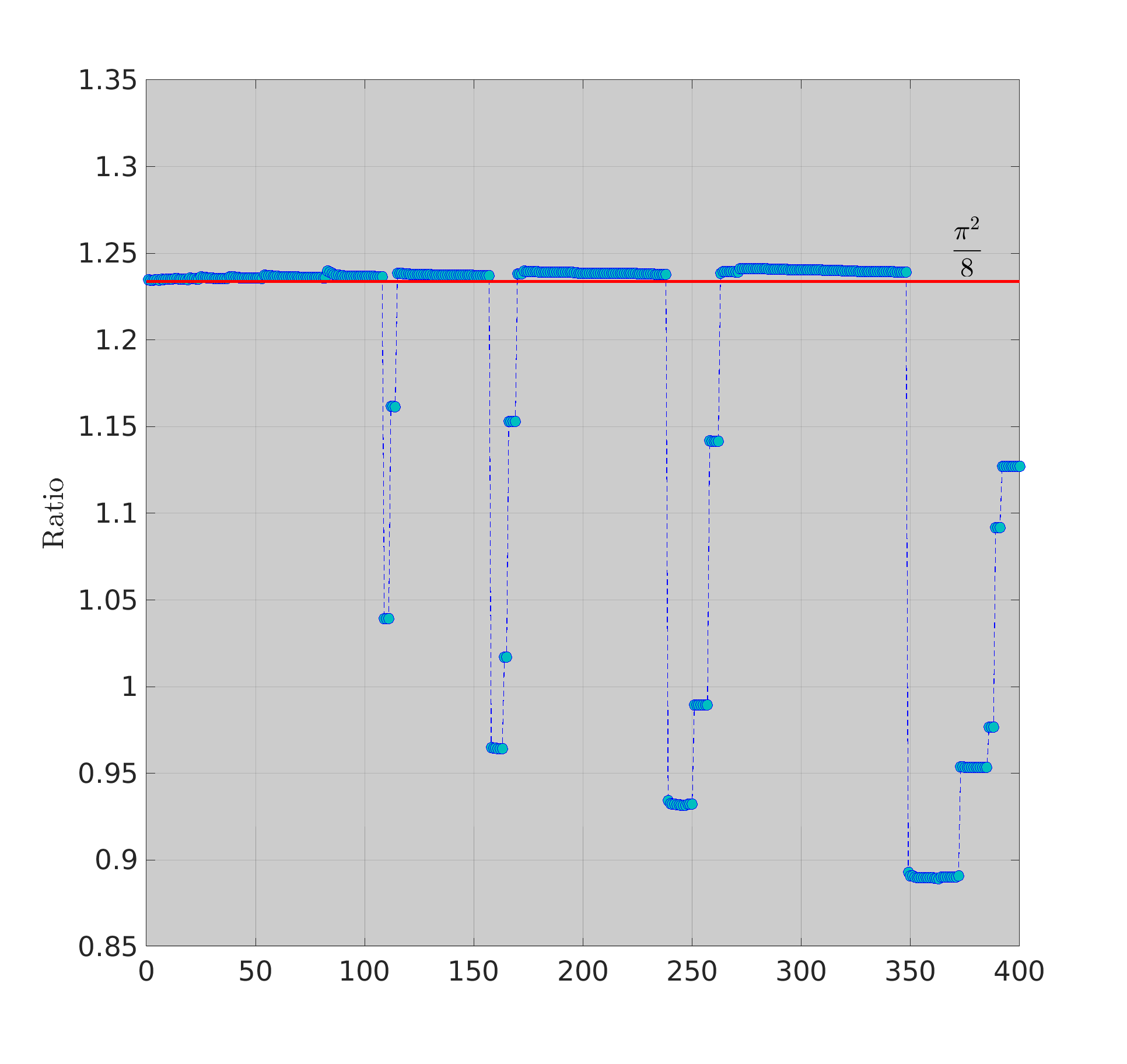}
	}
	\caption{The left plot displays a comparison of the first 400 eigenvalues with the corresponding local minima.  The right plot displays the corresponding ratio.  }	
	\label{Emore}
\end{figure}

Obviously, there are some pairs $(\lambda_n,W_n)$, whose ratio is away from  $\dfrac{\pi^2}{8}$. For example, in Figure \ref{Emore}, $(\lambda_{109},W_{109})$ is the first 'bad' pair. 

As we have introduced in Section \ref{sec:excited},  some bottom eigenvalues can be approximated by the first harmonics. However, $\lambda_{109}$ in Figure \ref{Emore} is actually contributed by the second harmonics, which means $W_{109}$ is not the correct associated local minimum.

In fact, if we consider higher up energy in \eqref{eq:Eis} contributed by the second, third, etc harmonics, there are no associated local minima from $W$ directly. To address the mismatch, we can construct a generalized local minima set, in which some artificial local minima are added. Take $\lambda_{109}$ in Figure \ref{Emore}
as an example: $\lambda_{109}$ is actually contributed by the second eigenvalue from the largest well. By (\ref{eq:Eis}), it should be almost 4 times the ground state eigenvalue from the largest well. There is no associated local minimum, but we can supplement one artificially: give it $4W_{\min}$ based on $\dfrac{\lambda_1}{W_{\min}}\approx\dfrac{\pi^2}{8}$,
then we may expect $\dfrac{\lambda_{109}}{4W_{\min}}\approx\dfrac{\pi^2}{8}$.

Therefore, we could construct a generalized local minima set of $W$. Let $W^{(1)}$ be the initial local minima set of the effective potential $W$, and the elements of  $W^{(1)}$ are sorted in  ascending order, i.e.
\[W^{(1)}=[W_1~W_2~\cdots].\]
Then we update the set by combining  $W^{(1)}$ and  $2^2W^{(1)}$. Specifically, let
\[\wt W^{(2)}=\left[
\begin{matrix}
 W^{(1)}      \\
2^2 W^{(1)}       \\
\end{matrix}
\right]=\left[
\begin{matrix}
 W_1      & W_2      & \cdots       \\
 2^2W_1      & 2^2W_2      & \cdots       \\
\end{matrix}
\right]\]
and get $W^{(2)}$ by sorting all the elements of $\wt W^{(2)}$  in ascending order:
\[W^{(2)}=\text{sort}(\wt{W}^{(2)}).\]
Similarly, for a positive integer $s$,  we could construct $W^{(s)}$ as
\[\wt W^{(s)}=\left[
\begin{matrix}
 W^{(1)}      \\
2^2 W^{(1)}       \\
3^2 W^{(1)} \\
 \vdots \\
 s^2   W^{(1)}     \\
\end{matrix}
\right]=\left[
\begin{matrix}
 W_1      & W_2      & \cdots   & \cdots     \\
 2^2W_1      & 2^2W_2      & \cdots  & \cdots      \\
  3^2W_1      & 3^2W_2      & \cdots  & \cdots      \\
 \vdots & \vdots & \ddots  & \vdots \\
 s^2W_1      & s^2W_2      & \cdots   & \cdots     \\
\end{matrix}
\right]\]
and sort all the elements of $\wt W^{(s)}$  in ascending order:
\[W^{(s)}=\text{sort}(\wt{W}^{(s)}).\]

Then, for a sufficiently large $s$, \eqref{eq:appro} is  modified as:
\[\lambda_n\approx\dfrac{\pi^2}{8} W^{(s)}_n,\]
where $W^{(s)}_n$ is the $n$-th element of $W^{(s)}$. 
But in practical, if we only focus on the first few eigenvalues, a mild $s$ and the associated $W^{(s)}$ are enough. For instance, we repair Figure \ref{Emore} by using $W^{(2)}$ and $ W^{(3)}$, instead of the initial $W^{(1)}$ shown in Figure \ref{Emore}.  We first apply $W^{(2)}$ in Figure \ref{Emore_w2}:

	\begin{figure}[H]
	\centering
	\subfigure{
		\includegraphics[width=7.2 cm]{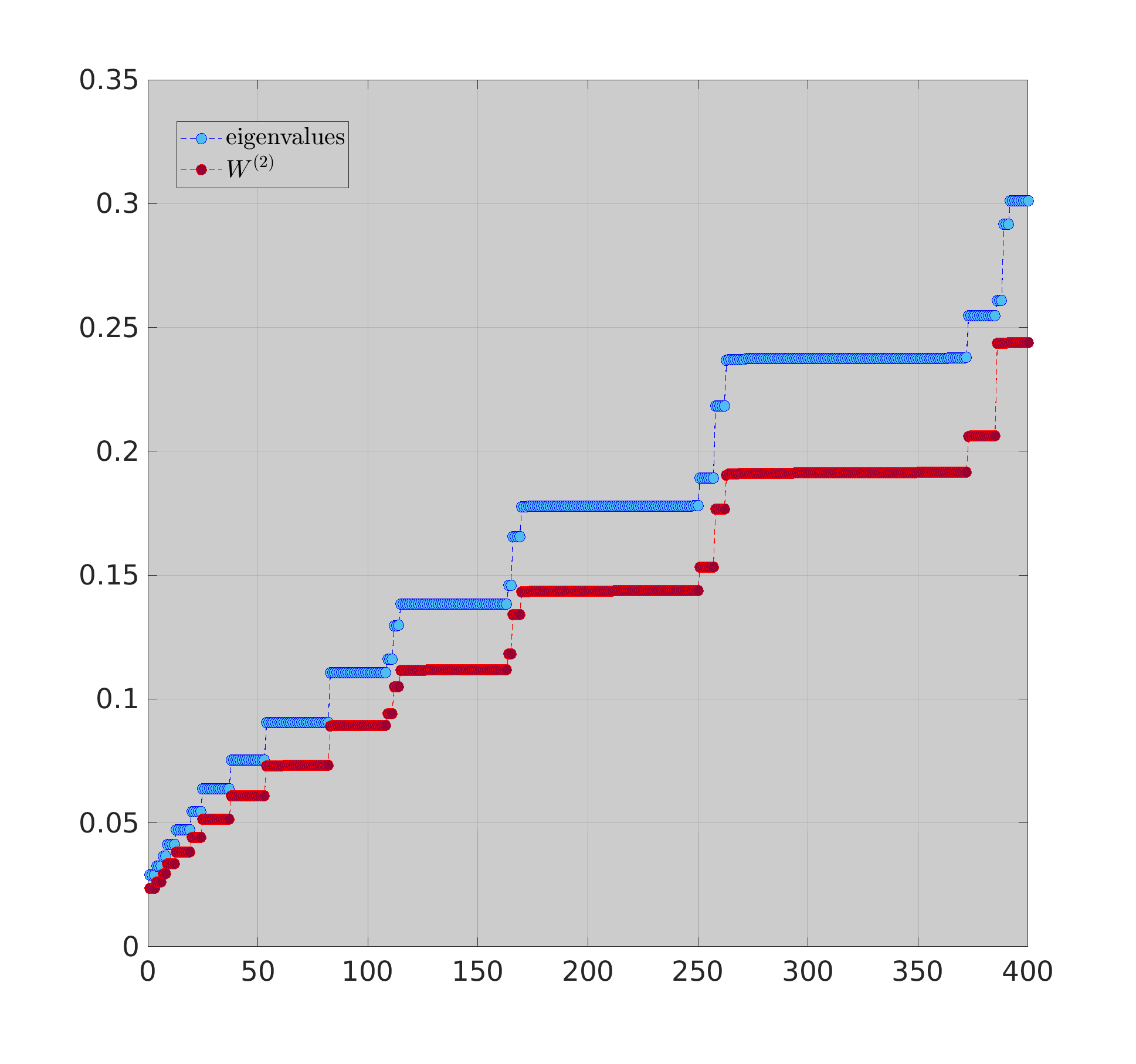}
		
	}
	\quad
	\subfigure{
		\includegraphics[width=7.2 cm]{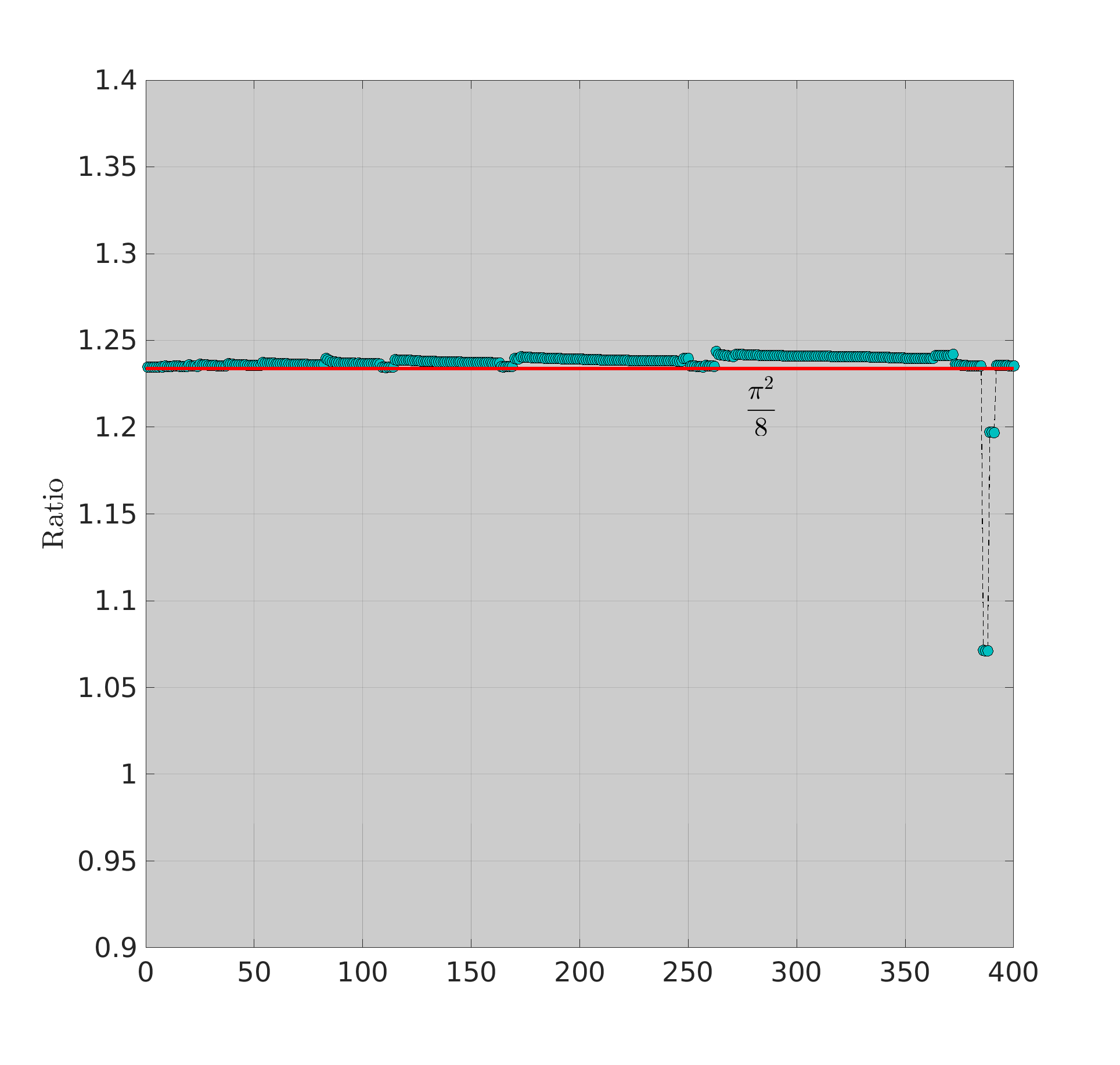}
	}
	\caption{The left plot displays a comparison of the first 400 eigenvalues with the first 400 values from $W^{(2)}$ .  The right plot displays the corresponding ratio $\dfrac{\lambda_n}{W^{(2)}_n}$, $n=1,2,\cdots,400$.  }	
	\label{Emore_w2}
\end{figure}
When $W^{(2)}$ is applied, the behavior of the ratio improves. However, $W^{(2)}$ can not repair all the ratio about the first 400 eigenvalues, because some eigenvalues are actually contributed by the third harmonics. Then we consider $W^{(3)}$ (Figure \ref{Emore_w3}):

	\begin{figure}[H]
	\centering
	\subfigure{
		\includegraphics[width=7.2 cm]{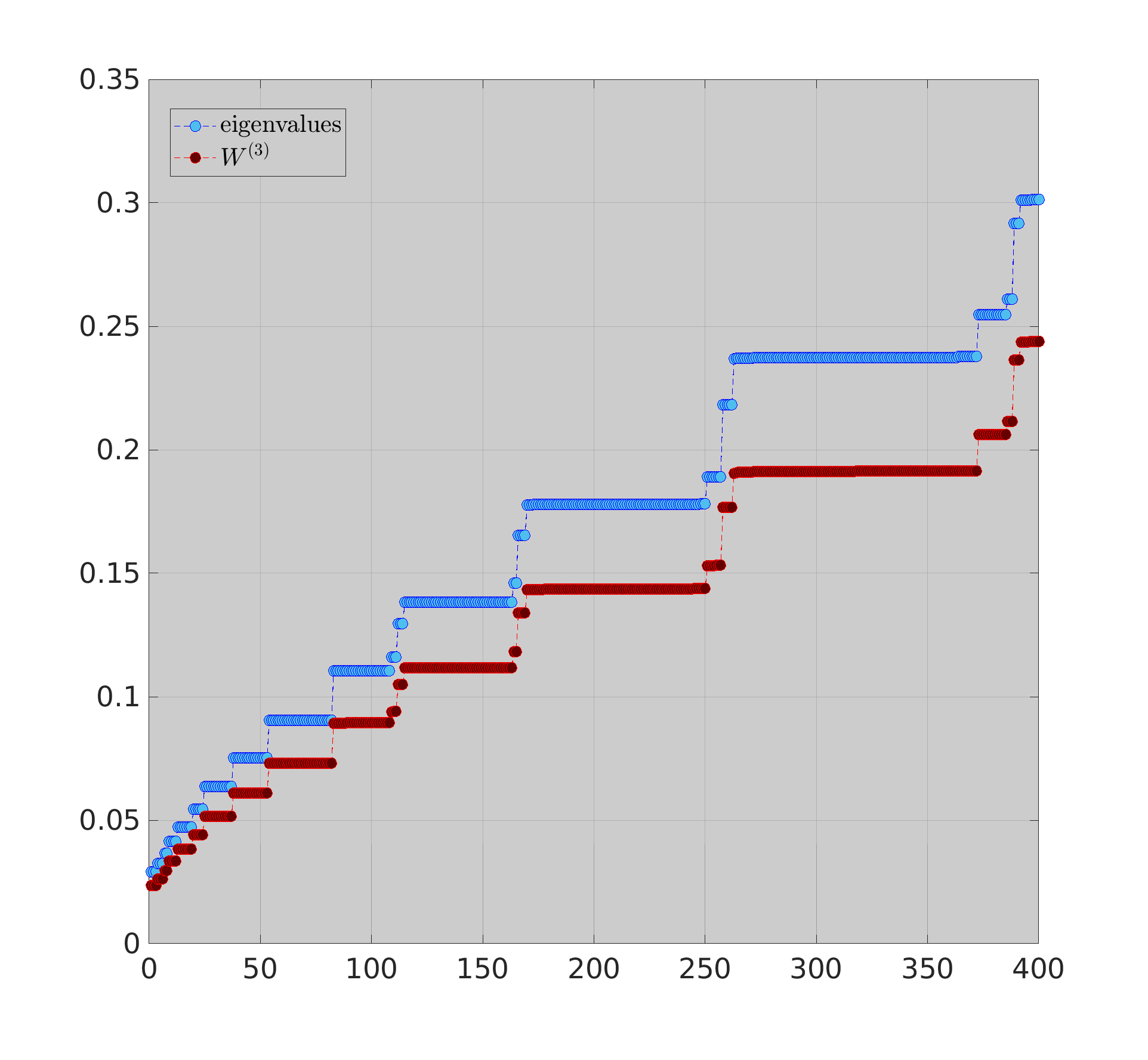}
		
	}
	\quad
	\subfigure{
		\includegraphics[width=7.2 cm]{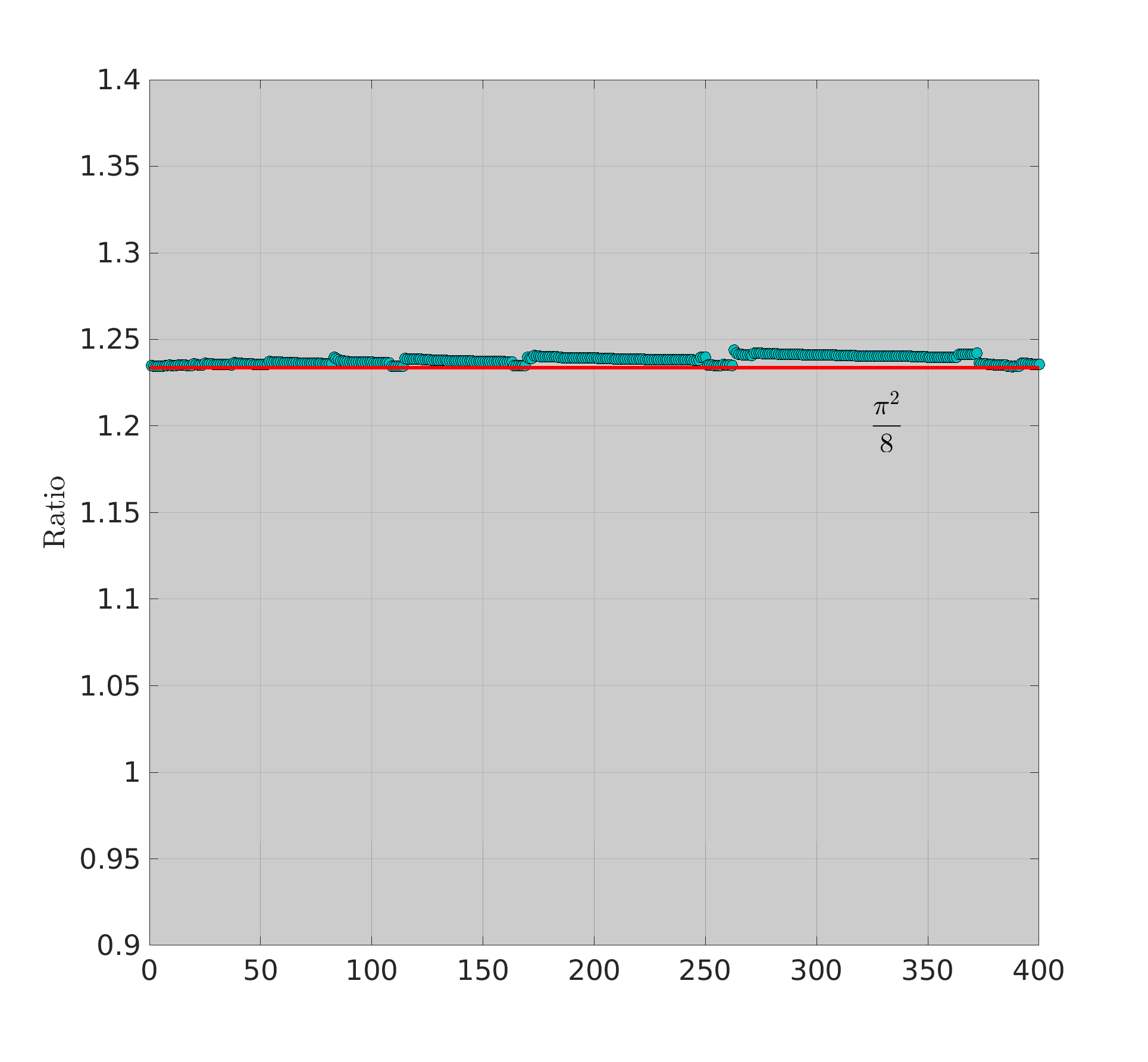}
	}
	\caption{The left plot displays a comparison of the first 400 eigenvalues with the first 400 values from $W^{(3)}$ .  The right plot displays the corresponding ratio $\dfrac{\lambda_n}{W^{(3)}_n}$, $n=1,2,\cdots,400$.  }		
	\label{Emore_w3}
\end{figure}

After $W^{(3)}$ is applied, we could finally see

\[\dfrac{\lambda_n}{W^{(3)}_n}\approx\dfrac{\pi^2}{8},\quad n=1,2,\cdots,400.\]
In other words, $\dfrac{\pi^2}{8}W^{(3)}_n$ ($n=1,2,\cdots,400$) could be used to approximate the first 400 eigenvalues efficiently.

On the other hand, $W^{(3)}$ is enough when we concentrate the first 400 eigenvalues in this case. This is because the fourth harmonics makes no contribution to any of the  first 400 eigenvalues. Actually, the first 400 values of $W^{(3)}$ will not change when it is updated to $W^{(4)}$.

The case in Figure \ref{Emore} is based on $V_{\max}=20$. Although it is not very high, it still works well after we apply the generalized local minima set $W^{(3)}$. Now we try a smaller $V_{\max}$. In the following case,  the potential $V_{\om}$ involves 30\% 4 and 70\% 0, and the domain size $L=1000000$. Then Figure \ref{fig10} shows the ratio of the first 400 eigenvalues over the first 400 values from $W^{(1)}$ and $W^{(3)}$.

\begin{figure}[H]
	\centering
	\includegraphics[width=0.7\linewidth]{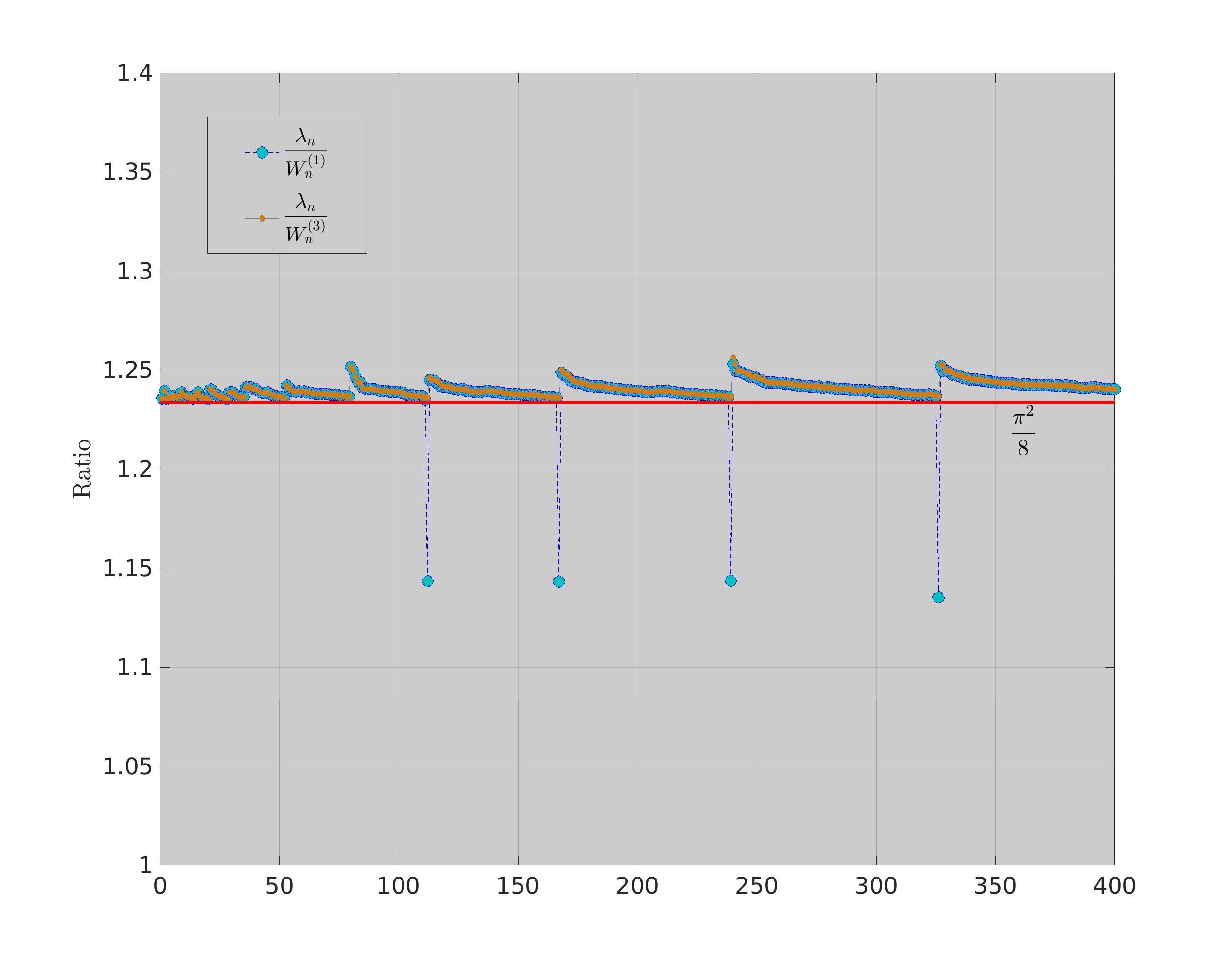}
	\caption{Comparison of  $\dfrac{\lambda_n}{W^{(1)}_n}$ and $\dfrac{\lambda_n}{W^{(3)}_n} $, $n=1,2,\cdots,400.$}
	\label{fig10}
\end{figure}

Consequently, it works well when we apply  $W^{(3)}$.

\vspace{1cm}

\noindent
\textbf{Acknowledgments.} The authors would like to thank Douglas N. Arnold and Svitlana Mayboroda for many stimulating discussions and useful suggestions. 

Chenn is supported through a Simons Foundation Grant (601948 DJ) and a PDF fellowship from NSERC/Cette recherche a \'{e}t\'{e} financ\'{e}e par le CRSNG.  Wang is supported by Simons Foundation grant 601937, DNA. Zhang
is supported in part by the NSF grants DMS1344235, DMS-1839077, and Simons Foundation grant 563916, SM.

\Addresses

\end{document}